\documentclass[11pt]{amsart}
\usepackage{amsmath,amsthm,amsfonts,amssymb,times}

\usepackage{listings}
\usepackage{color}

\definecolor{dkgreen}{rgb}{0,0.6,0}
\definecolor{gray}{rgb}{0.5,0.5,0.5}
\definecolor{mauve}{rgb}{0.58,0,0.82}

\numberwithin{equation}{section}  

\theoremstyle{remark}

\theoremstyle{plain}
\newtheorem{prop}{Proposition}
\newtheorem{lem}{Lemma}[section]
\newtheorem{thm}{Theorem}

\newcommand{\ZZ}{{\mathbb Z}}

\newcommand{\RR}{{\mathbb R}}
\newcommand{\CC}{{\mathbb C}}
\newcommand{\NN}{{\mathbb N}}

\newcommand{\cB}{\ensuremath{\mathcal{B}}}

\newcommand{\cN}{\ensuremath{\mathcal{N}}}

\newcommand{\cQ}{\ensuremath{\mathcal{Q}}}
\newcommand{\cR}{\ensuremath{\mathcal{R}}}

\newcommand{\cU}{\ensuremath{\mathcal{U}}}
\newcommand{\cT}{\ensuremath{\mathcal{T}}}
\newcommand{\cV}{\ensuremath{\mathcal{V}}}

\newcommand{\cY}{\ensuremath{\mathcal{Y}}}

\newcommand{\PR}{\mathbb{P}}
\newcommand{\E}{\mathbb{E}}

\newcommand{\e}{\ensuremath{\mathrm{e}}}

\makeatletter
\renewcommand{\pmod}[1]{\allowbreak\mkern7mu({\operator@font mod}\,\,#1)}
\makeatother

\newcommand{\dalign}[1]{\[\begin{aligned} #1 \end{aligned}\]}

\newcommand{\be}{\begin{equation}}
\newcommand{\ee}{\end{equation}}

\newcommand{\ssum}[1]{\sum_{\substack{#1}}}  
\newcommand{\sprod}[1]{\prod_{\substack{#1}}}  

\newcommand{\lam}{\ensuremath{\lambda}}

\renewcommand{\a}{\ensuremath{\alpha}}
\renewcommand{\b}{\ensuremath{\beta}}
\newcommand{\del}{\ensuremath{\delta}}
\newcommand{\eps}{\ensuremath{\varepsilon}}

\renewcommand{\le}{\leqslant}

\renewcommand{\ge}{\geqslant}

\newcommand{\fl}[1]{{\ensuremath{\left\lfloor {#1} \right\rfloor}}}  
\newcommand{\order}{\asymp}      
\renewcommand{\(}{\left(}
\renewcommand{\)}{\right)}
\newcommand{\ds}{\displaystyle}
\newcommand{\pfrac}[2]{\left(\frac{#1}{#2}\right)}  

\newcommand{\blam}{\ensuremath{\boldsymbol{\lambda}}}
\newcommand{\qq}{\mathbf{q}}
\newcommand{\MM}{\mathbf{M}}

\newcommand{\bn}{\ensuremath{\binom{2n}{n}}}
\DeclareMathOperator{\Ein}{Ein}

\begin{document}

\title{Divisibility of the central binomial coefficient $\binom{2n}{n}$}
\author{Kevin Ford}
\address{Department of Mathematics, 1409 West Green Street, University
of Illinois at Urbana-Champaign, Urbana, IL 61801, USA}
\email{ford126@illinois.edu}
\author{Sergei Konyagin}
\address{Steklov Mathematical Institute,
8 Gubkin Street,
Moscow, 119991, Russia}
\email{konyagin23@gmail.com}

\begin{abstract}
We show that for every 
fixed $\ell\in\NN$, the set of $n$ with $n^\ell|\binom{2n}{n}$ has
a positive asymptotic density $c_\ell$ , and we give an asymptotic
formula for $c_\ell$ as $\ell\to\infty$.
We also show that $\# \{n\le x, (n,\bn)=1 \} \sim cx/\log x$ 
for some constant $c$.
 We use 
results about the anatomy of integers and tools from 
Fourier analysis.  One novelty is a method to
capture the effect of large prime factors of integers
in general sequences.
\end{abstract}

\date{\today}
\thanks{The first author was supported in part by National
Science Foundation Grant  DMS-1802139}
\thanks{The authors are grateful to the Institute of Mathematics of the Bulgarian Academy of Sciences, which hosted their visit in June, 2018, and where the seeds of this paper were sown.  The first author thanks the Institute of Mathematics at the University of  Oxford
  for providing stimulating working conditions during a visit in March-June, 2019.}
\thanks{The authors thank the anoymous referee for many helpful comments.}

\maketitle


\section{Introduction}

That $(n+1)|\binom{2n}{n}$ for every positive integer $n$ is
a consequence of the integrality of  the Catalan numbers.
In \cite{Pom}, Pomerance raised the question of how
frequently $n+k|\binom{2n}{n}$, where $k$ is a fixed integer.
Pomerance showed with a simple argument that when $k$ is positive, 
almost all $n$ have the property  $n+k|\binom{2n}{n}$,
and the exceptional set up to $x$ is $O(x^{1-a_k})$ for
some $a_k>0$. When $k\le 0$, he proved that the set of such $n$ is 
governed by the set of such $n$ corresponding to $k=0$;
more precisely,
\[
\# \bigg\{ n\le x: (n+k) \big|\bn \bigg\} = \# \bigg\{n\le x: n\big | \bn\bigg\} + O(x^{1-a_k}).
\]
  Pomerance conjectured 
that $n|\bn$ on a set of positive lower density,
and showed that it has upper density at most $1-\log 2$;
this is an easy consequence of the fact that if $n$ has 
a prime factor larger than $\sqrt{2n}$, then  $n\nmid \bn$.
The upper asymptotic  density was later improved by Sanna  \cite{sanna} to 
$ \le 1-\log 2 -0.0551$.

Divisibility of $\binom{2n}{n}$ by $n^\ell$
has also been considered by several people; see the On-line Encyclopedia of Integer Sequences \cite{OEIS}, sequences
A014847 ($\ell=1$), A121943 ($\ell=2$), A282163 ($\ell=3$), A282346 (smallest $n>1$ with $n^\ell|\bn$, $\ell\ge 1$), A282672 ($\ell=6$), A283073 ($\ell=4$),
and A283074 ($\ell=5$).

Our main result is the following.

\begin{thm}\label{main_theorem}
Fix $\ell\in\NN$.  The set of $n$ with $n^\ell|\binom{2n}{n}$ has
a positive asymptotic density $c_\ell$.  The density may be
computed as follows: Let $U_1,U_2,\ldots$ be independent
uniform-$[0,1]$ random variables, and let
\be\label{kj}
g_1=\fl{\frac{1}{U_1}} -1, \; g_2=\fl{\frac{1}{(1-U_1)U_2}}-1, \;
\ldots, \; g_j=\fl{\frac{1}{(1-U_1)\cdots (1-U_{j-1})U_j}}-1, \ldots.
\ee
Then
\[
c_\ell=\E \prod_{j=1}^\infty \(1-2^{-g_j}\sum_{h=0}^{\ell-1} 
\binom{g_j}{h} \).
\]
\end{thm}

In Table \ref{table1}, we list counts for the
number of $n$ in various intervals with $n^\ell|\bn$,
$1\le \ell \le 3$, and compare with the
theoretical limiting densities coming from
Theorem \ref{main_theorem} (truncated to five significant decimal places).  The tabulation of $n$ such that
$n^\ell | \bn$ was performed by two programs written
by the authors, one in the C language and the other in
PARI-GP, the latter being slower but applicable for the
larger ranges beyond $10^{17}$.  The numbers for
$[1,10^k]$, $k\le 8$, were run by both programs and
agreed exactly.  These counts also agree with data
gathered by Giovanni Resta (personal communication),
who has also provided the data for $[1,10^{11}]$.

  See Section \ref{sec:numerics} for details of the calculation
  of the densities and reasons why we believe the calculations to be accurate to the decimal places displayed.
 It is evident from Table \ref{table1} that the convergence
  to the limit $c_\ell$ is very slow.

   \begin{center}
   \renewcommand{\arraystretch}{1.25}
        \begin{table}
        \begin{tabular}{|r|r|r|r|} \hline
        Interval     & $\ell=1$ & $\ell=2$ & $\ell=3$    \\ \hline 
       $[1,10^5]$ & 11,360 & 193 & 1 \\
       $[1,10^6]$ & 118,094 & 2,095 & 3 \\
       $[1,10^7]$ &   1,211,889 & 23,921 & 67 \\
       $[1,10^8]$ & 12,325,351 & 279,042 & 1,055 \\
       $[1,10^9]$ & 123,795,966 & 2,994,447 & 12,968 \\
       $[1,10^{10}]$ & 1,240,345,721 & 31,983,305 & 172,498 \\
       $[1,10^{11}]$ & 12,383,984,058 & 332,839,293 & 2,031,901 \\
       $(10^{17},10^{17}+10^8]$ & 12,169,463 & 364,815 & 3,390 \\
       $(10^{30}, 10^{30}+10^7]$ & 1,180,797 & 34,734 & 351\\ 
       \hline 
       $c_\ell$ & 0.11424 & 0.0032277 & 0.000031511 \\
       \hline 
        \end{tabular}

        \medskip
        
        \caption{Numerical counts vs. theoretical limits, 
        $1\le \ell\le 3$}
        \label{table1}
        \end{table}
\end{center}

\begin{thm}\label{clasym}
We have 
\[
c_\ell \sim \rho\( 2\ell + 1 - \log(2\ell \log (2\ell)) - \frac{\log\log (2\ell)}{\log 2\ell}\),
\]
as $\ell\to\infty$, where $\rho$ is the Dickman function.
\end{thm}

The Dickman function $\rho$ is the unique continuous
solution of the 
differential-delay equation
\be\label{rho-diff}
\rho(u)=1 \quad (u\le 1), \quad -u\rho'(u)=\rho(u-1) \quad (u>1).
\ee
Roughly, $\rho(u)$ decays like $1/\Gamma(u)$, and in fact 
$\rho$ is strictly decreasing for $u>1$ and
\be\label{rho-rough}
\rho(u) = e^{-u(\log u + \log\log u + O(1))}.
\ee
Given Theorem \ref{main_theorem}, a rought heuristic for the values given in Theorem \ref{clasym} is that the factor
\[
1-2^{1-g_j}\sum_{h=0}^{\ell-1} 
\binom{g_j-1}{h}
\]
is close to 1 when $g_j$ is substantially larger than $2\ell$ 
and is close to 0 when $g_j$ is substantilly smaller than $2\ell$.
Thus, $c_\ell$ should be close to the probability that $g_j\ge 2\ell$
for all $j$, which equals $\rho(2\ell)$.

A related problem is the study of the set $\cB$ of positive integers $n$ such
that $n$ and $\binom{2n}{n}$ are coprime, see e.g. sequence
A082916 of the OEIS \cite{OEIS}.
In \cite{sanna}, Sanna showed that 
$\# (\cB \cap [1,x]) \ll x/\sqrt{\log x}$ for all $x>1$.
On the other hand, $\cB$ contains all odd primes, and thus
$\# (\cB \cap [1,x]) \ge (1+o(1))x/\log x$  for all $x\ge2$.
We sharpen these results by proving an asymptotic formula for
$\# (\cB \cap [1,x]) $.

\begin{thm}\label{thm:coprime}
We have $\# \{n\le x : (n,\bn)=1 \} \sim c x/\log x$
as $x\to\infty$,
where 
\be\label{cdef}
c=\sum_{k=1}^\infty \; \frac{1}{k!}\;\; 
\idotsint\limits_{\substack{u_i\ge 0\; \forall i \\u_1+\cdots+u_k=1}} 
h(u_1)\cdots h(u_k)\, du_1\cdots du_{k-1}, \qquad
h(x) = x^{-1} 2^{1-\fl{1/x}}.
\ee
\end{thm}

As $h$ is bounded, the series for $c$ converges rapidly.
Numerically, $c=1.526453\ldots$ (See section \ref{sec:c}).
This is also a good match to numerical data, see
Table \ref{table2}. 

 \begin{center}
   \renewcommand{\arraystretch}{1.25}
\begin{table}[h]
  \begin{tabular}{|r|r|r|}
    \hline
    $x$ & $N$ & $\frac{N}{x/\log x}$ \\
    \hline
    $10^4$ & 1734 & 1.597073 \\
    $10^5$ & 13487 & 1.552748 \\
    $10^6$ &   111460 &  1.539876\\
    $10^7$ &   950039  & 1.531281\\
     $10^8$ & 8282970 &   1.525779\\
    $10^9$ &  73631430  &     1.525883\\
    $10^{10}$ &   662319904 &  1.525047\\
    $10^{11}$ & 6022446576  &  1.525391 \\
    \hline
  \end{tabular}
\caption{Number, $N$, of integers $\le x$ with $(n,\bn)=1$}\label{table2}
\end{table}
\end{center}

%
%
%
\subsection{Heuristics}
%
%
For most $n$, the divisibility condition $n^\ell|\binom{2n}{n}$
is essentially determined by the largest prime factors of $n$.
By Kummer's criterion \cite{K},
if $p$ is prime, then $p^{\ell}|\binom{2n}{n}$ if and only if
the addition of $n$ and $n$ in base-$p$ has at least $\ell$ carries.  
This is equivalent to $\{n/p^s\}>\frac12$ for at least $\ell$
values of $s\in \NN$.
If $p$ is large, then this means (essentially) that 
the base-$p$ expansion of $n$ has at least $\ell$ digits which are
$\ge \frac{p-1}{2}$ (if a digit equals $\frac{p-1}{2}$, then
it may or may not induce a carry).
Supposing that $p\| n$, 
the final base-$p$ digit is zero, and the leading digit is
$<p/2$ with high probability.  There are $k=\fl{\frac{\log n}{\log p}}-1$ 
remaining base-$p$ digits, and if these are randomly distributed (over 
all $n\le x$ divisible by $p$ and not by $p^2$) then we 
expect that $p^\ell | \binom{2n}{n}$ occurs with probability close to
\[
1-2^{1-k}\sum_{h=0}^{\ell-1}  \binom{k-1}{h}.
\]

Donelly and Grimmett \cite{DG}  (see also \cite{Ten})
proved that the largest prime factors of a random integer
have, asyptotically, the Poisson-Dirichlet distribution.
A realization of this distribution is 
given  in terms of independent uniform-$[0,1]$ random variables
$U_1, U_2, \ldots$.
 Let $(X_1,X_2,\ldots)$ be the infinite
dimensional vector formed from the decreasing rearrangement of the
numbers \be\label{Un} Y_1=U_1, Y_2=(1-U_1)U_2,
Y_3=(1-U_1)(1-U_2)U_3, \ldots. \ee 
Then $(X_1,X_2,\ldots)$ has the Poisson-Dirichlet distribution.
Let $p_j(n)$ denote the $j$-th largest prime factor of $n$.
 The paper \cite{DG} gives a simple, transparent proof
that $(X_1,\ldots,X_k)$ and 
\[
\(\frac{\log p_1(n)}{\log n}, \ldots, \frac{\log p_k(n)}{\log n}\)
\]
have identical distributions (asymptotically as $x\to \infty$,
where $n$ is drawn at random from $[1,x]$).
For a discussion of other realizations of the Poisson-Dirichlet 
distribution, see Section 1 of \cite{Ten}.
Combining this with our heuristic above about divisibility
of $\binom{2n}{n}$ by $p^\ell$, we arrive at 
Theorem \ref{main_theorem}.

The heuristic for Theorem \ref{thm:coprime} is simpler.
If $n$ has $k$ prime factors $p_1,\ldots,p_k$, 
with $p_i=x^{u_i}$, then
we expect $(n,\bn)=1$ with probability $\prod_{i=1}^k
2^{1-\fl{1/u_i}}$.  Summing over all $p_1,\ldots,p_k$
with the prime number theorem yields the result in Theorem \ref{thm:coprime}.

We will make both of these heuristics precise utilizing harmonic analysis to detect the simultaneous divisibility of $\bn$ by large prime factors 
of $n$.
Section \ref{sec:expsum} contains the relevant estmates.
In Section \ref{sec:smallprimes}, we show that the small
prime factors of $n$ divide $\bn$ with very high probability,
and can safely be ignored.  We prove a result 
about simultaneous fractional parts of quotients of primes
in Section \ref{sec:fracparts} that will be needed
for Theorems \ref{main_theorem} and \ref{thm:coprime}.
  The proof of 
Theorem \ref{main_theorem} occupies Section
\ref{sec:thm1} and  we prove Theorem \ref{thm:coprime} in Section \ref{sec:coprime}.
Sections \ref{sec:numerics} and \ref{sec:clasym}
are devoted to the study of the constants $c_\ell$,
culminating in the proof of Theorem \ref{clasym}.
Finally, we desribe how to compute $c$ accurately in 
Section \ref{sec:c}.

%
%
%
{\Large \section{small prime factors}\label{sec:smallprimes}}
%
%
%
%

In this section, we will see that only the largest
prime factors of $n$ matter for Theorems \ref{main_theorem} and
\ref{thm:coprime}.

\begin{lem}\label{smallprpwr}
Let $p$ be prime, $v\in \NN$, $\ell\in \NN$ and 
$p^{\ell v} \le x^{1/100}$. Then
\[
\# \Big\{n\le x : p^v|n, \, p^{\ell v}\nmid \bn \Big\} \ll
\frac{x^{1-\frac{1}{3\log p}}}{p^v} e^{v/3}.
\]
\end{lem}

\begin{proof}
Suppose that $n\le x$ and $p^v|n$.   Write $n$ in
base-$p$ as $n=(b_D b_{D-1}\cdots b_0)_p$, where
$D=\fl{\frac{\log x}{\log p}}$, so that $b_0=\cdots=b_{v-1}=0$.
Also observe that the hypotheses imply that $D\ge 100v$ and hence that
\[
\ell v \le \frac{\log x}{100\log p} \le \frac{D+1}{100}
< \frac{D}{99} \le \frac{D-v}{98}.
\]
 The number
of choices for $b_{D}$ is at most $x/p^D$.
 By Kummer's criterion,
if $p^{\ell v}\nmid \bn$, then at most $\ell v-1$ of the digits
$b_{v},\ldots,b_{D-1}$ are $\ge \frac{p}2$. Hence, the number
of choices for $(b_{v},\ldots,b_{D-1})$ is at most
\[
\sum_{j=0}^{\ell v-1} \binom{D-v}{j} \pfrac{p-1}2^j \pfrac{p+1}{2}^{D-v-j}
\ll \pfrac{p+1}{2}^{D-v} \binom{D-v}{\ell v}
\]
if $p\ge 3$, and $O(\binom{D-v}{\ell v})$ when $p=2$.
Recalling that $\ell v \le (D-v)/98$,
by Stirling's formula we have
\[
\binom{D-v}{\ell v} \ll e^{0.057(D-v)}
\]
and thus 
\[
\# \Big\{n\le x : p^v|n, p^{\ell v}\nmid \binom{2n}{n} \Big\} \ll
\frac{x}{p^v} \pfrac{e^{0.057}(1+\frac13)}{2}^{D-v}
\ll \frac{x}{p^v} e^{-(D-v)/3},
\]
and the claimed inequality follows.
\end{proof}

\begin{prop}\label{smallprfact}
For large $x$, let $\del$ satisfy
$0 < \del \le 1$. For any $1\le n\le x$,
write $n=A_nB_n$, where $P^+(A_n)\le x^{\del} <P^{-}(B_n)$.  Fix $\ell\in\NN$.
Then
\[
\# \Big\{n\le x : A_n^\ell \nmid \binom{2n}{n} \Big\} \ll_\ell x \e^{-1/(300\ell \del)}.
\] 
\end{prop}

\begin{proof}
We may assume that $\frac{\log 2}{\log x} < \del \le 1/(300\ell)$, else the statement
is trivial.  
Hence, by Lemma \ref{smallprfact},
\dalign{
\# \Big\{n\le x : A_n^\ell \nmid \binom{2n}{n} \Big\} &\le 
\sum_{p\le x^\del} \left[ \sum_{v\le \frac{\log x}{100\ell \log p}}\# \Big\{n\le x : p^v|n, p^{\ell v} \nmid \bn \Big\} + \sum_{v>   \frac{\log x}{100\ell \log p}} \frac{x}{p^v} \right]\\
&\ll \sum_{p\le x^\del} \Bigg[ x^{1-1/(3\log p)} 
\sum_{v\le \frac{\log x}{100\ell \log p}} \frac{e^{v/3}}{p^v} +
x^{1-\frac{1}{100\ell}} \Bigg] \\
&\ll x^{1+\del-\frac{1}{100\ell}} + x \sum_{p\le x^\del} \frac{x^{-1/(3\log p)}}{p} \\
&\ll x^{1-\frac{1}{150\ell}} + x e^{-\frac{1}{3\del}}\\
&\ll x e^{-\frac{1}{300\ell \del}}.
}
\end{proof}

Next, we prove analogous bounds for integers with a given smallest prime factor.

\begin{prop}\label{coprime-smpr}
The number of integer $n\le x$ for which $(n,\bn)=1$
and $n$ has a prime factor smaller than $n^{\delta}$
is $O(\frac{x}{\log x} e^{-1/(3\delta)})$.
\end{prop}
\begin{proof}
Fix $p$ and consider those $n$ with smallest prime factor $p$
and such that $p\nmid \bn$.
We argue as in the $\ell=1$ case of Lemma \ref{smallprfact},
except that
for fixed $b_2,\ldots,b_{D}$ we bound the number of possible 
$b_1$ such that $\sum_{j=1}^D p^j b_j$ has no prime factor less than $p$ with a sieve (e.g., \cite[Theorem 2.2]{HR}), obtaining
\[
\# b_1 \ll \frac{p}{\log p}.
\]
It follows that
\[
\# \Big\{ n\le x : n \text{ has smallest prime factor }p, p\nmid \bn  \Big\} \ll \frac{x^{1-\frac{1}{3\log p}}.}{p\log p} .
\]
Summing over $p\le x^{\delta}$ completes the proof.
\end{proof}

%
%
%
{\Large \section{Exponential sum estimates}\label{sec:expsum}}
%
%
%
%

We gather together in this section various estimates for
exponential sum which we will need for the proof of Theorem
\ref{main_theorem}.

The first lemma is the 'Weyl-van der Corput inequality'
(see Theorems 2.2, 2.8 in \cite{GK}).  
It is far from the best result of its kind, but has a relatively short proof and suffices for
our purposes.

\begin{lem}\label{Weyl}
Let $j\ge 2$ be an integer, let $I$ be an interval and suppose that $f\in C^{j}(I)$ and that
\[
\lam \le |f^{(j) }(x)| \le \a \lam
\]
where $\lam>0$, $\a\ge 1$.  Then
\[
\sum_{n\in I} e(f(n)) \ll |I|(\a^2\lam)^{\frac{1}{4J-2}}+|I|^{1-\frac{1}{2J}} \a^{\frac{1}{2J}}  +|I|^{1-\frac{2}{J}+\frac{1}{J^2}}\lam^{-\frac{1}{2J}},
\]
where $J=2^{j-2}$.
\end{lem}

We apply this lemma to bound a certain class of exponential sums.

\begin{lem}\label{expsumf}
Let $N\in \NN$, and
\be\label{fu}
f(u) = \a u + \sum_{r=r_1}^{r_2} \frac{\beta_r}{u^r},
\ee
where $\a\in \RR$, $1\le r_1\le r_2$, and for some $A\in[1,N^{1/2}]$ we have
\be\label{betar}
|\b_{r_1}| \ge N^{r_1} A, \quad |\b_{r}/\b_{r_1}| \le N^{(r-r_1)/2} \; (r_1\le r\le r_2).
\ee
Then
\[
\max_{I \subset (N,2N]} \sum_{n\in I} e(f(n)) \ll_{r_2} N \( N^{-1/2^{j}} + A^{-1/4} \),
\]
where
\be\label{jdef}
j = 3 + \fl{\frac{\log\(\frac{|\b_{r_1}|}{ AN^{r_1}}\)}{\log N}}.
\ee
\end{lem}

\begin{proof}
We apply Lemma \ref{Weyl}.  Firstly, we may assume
that $N$ is sufficiently large and that 
\be\label{jrA}
j\le \frac{\log\log N}{\log 2},
\ee
 for otherwise the 
conclusion is trivial.  Also note that $j\ge 3$.
Denoting by $r^{(j)}$ the rising factorial $r(r+1)\cdots(r+j-1)$,
and using \eqref{betar},
we have for $N<u\le 2N$ the relation
\dalign{
f^{(j)}(u)&=(-1)^j \sum_{r=r_1}^{r_2} \frac{r^{(j)} \b_r}{u^{r+j}}\\
&=(-1)^j \frac{r_1^{(j)}\b_{r_1}}{u^{r_1+j}}\Bigg(1+ O\Bigg(
\sum_{r=r_1+1}^{r_2}  \frac{(r^{(j)}/r_1^{(j)}) |\b_r/\b_{r_1}|}
{N^{r-r_1}}\Bigg)\Bigg)\\
&= (-1)^j \frac{r_1^{(j)}\b_{r_1}}{u^{r_1+j}}\Bigg(1+ O\bigg( 
\sum_{r=r_1+1}^{r_2}  \frac{(r/r_1)^j}{N^{(r-r_1)/2}}\bigg)\Bigg) \\
&= \(1+O_{r_2}\(N^{-1/2}\)\) (-1)^j \frac{r_1^{(j)}\b_{r_1}}{u^{r_1+j}}.
}
For large enough $N$ it follows that
\[
\lam \le |f^{(j)}(u)| \le \a\lam, \quad 
\lam=\frac{r_1^{(j)}|\b_{r_1}|}{2(2N)^{r_1+j}},\quad \a= 2^{r_1+j+2}.
\]
Inserting this bound into
 Lemma \ref{Weyl}, we have
\be\label{sumefn}
\frac{1}{N} \sum_{n\in I} e(f(n)) \ll_{r_2}
\lam^{\frac{1}{4J-2}}+N^{-\frac{1}{2J}}  +N^{-\frac{2}{J}+\frac{1}{J^2}}\lam^{-\frac{1}{2J}},
\ee
where $J=2^{j-2}$.  We note that from \eqref{betar} and
the definition of $j$,
\[
N^2 \frac{|\beta_{r_1}|}{AN^{r_1}} \le N^j \le N^3 \frac{|\beta_{r_1}|}{AN^{r_1}}
\]
and hence that
\[
\frac{A}{2^{r_1+j+1}N^3} \le \lambda \le r_1^{(j)} \pfrac{A}{N^2} \le r_1^{(j)} N^{-3/2}.
\]
When $j=3$, therefore, the right side of \eqref{sumefn} is 
\[ 
\ll_{r_2} \lam^{1/6}+N^{-1/4}+N^{-3/4} \lam^{-1/4} \ll N^{-1/4}+A^{-1/4}.
 \]
Now assume that $j\ge 4$ so that $J\ge 4$.  Then the right side of
\eqref{sumefn} is 
\[
\ll_{r_2} N^{-\frac{3/2}{4J-2}}+N^{-\frac{1}{2J}}+ N^{-\frac{7}{4J}} (N^3)^{\frac{1}{2J}}\ll_{r_2} N^{-\frac{1}{4J}}.
\]
Combining the two cases, $j=3$ and $j>3$, this concludes the proof.
\end{proof}

We now apply Lemma \ref{expsumf} to bound analogous sums over primes.

\begin{lem}\label{expsumprimes}
Assume $f$ satisfies \eqref{fu}, where the coefficients satisfy 
\eqref{betar} for some $A\in [1,N^{1/6}]$.  Then 
\[
\max_{I\subset (N,2N]} \sum_{p\in I} e(f(p)) \ll_{r_2} N (\log N)^4
\big( N^{-\frac{1}{3\cdot 2^j}} + A^{-1/10} \big),
\]
where $j$ is given by \eqref{jdef}.
\end{lem}

\begin{proof}
Our technique is standard.  Throughout, constants implied by
$O-$ and $\ll$- may depend on $r_1,r_2$.  We begin by applying Vaughan's identity, taking $U=V=N^{1/3}$ 
in \cite[p. 139]{Dav}.  This gives
\be\label{sump-1}
 \sum_{p\in I} e(f(p)) = O(N^{1/2})+\sum_{n\in I} \Lambda(n)e(f(n)) = 
 O(N^{1/2}) + S_2 + S_3 + S_4,
\ee
where, following the notation from \cite{Dav} (observe that
$S_1$ is trivially zero in our case), we define
\begin{align*}
S_2 &= - \sum_{a\le N^{1/3}} \Lambda(a) \sum_{b\le N^{1/3}} \mu(b)
\ssum{abc\in I} e(f(abc)), \\
S_3 &= \sum_{b\le N^{1/3}} \sum_{bc\in I} \mu(b) (\log c) e(f(bc)), \\
S_4 &= \sum_{b>N^{1/3}} h(b) \ssum{bc \in I \\ c>N^{1/3}} \Lambda(c) e(f(bc)),
\end{align*}
where 
\[
h(b) = \ssum{d|b \\ d>N^{1/3}} \mu(d).
\]
We may apply Lemma \ref{expsumf} directly to $S_2$ and to $S_3$;
these are called ``Type I'' sums in the modern literature.
For $S_2$, we fix $a$ and $b$ and apply Lemma \ref{expsumf} with
$N$ replaced by $N/ab$ and $\beta_r$ replaced by 
$\beta_r/(ab)^r$.  We check that
\[
A \le N^{1/6} \le (N/ab)^{1/2}, \qquad \left|\frac{\beta_r'}{\beta_{r_1}'}\right| =  \left|\frac{\beta_r}{\beta_{r_1}}\right| (ab)^{-(r-r_1)}
\le \pfrac{N}{ab}^{(r-r_1)/2}.
\]
Thus, for any $a,b$ we have
\[
\ssum{abc\in I} e(f(abc)) \ll \frac{N}{ab} \big( (N/ab)^{-1/2^j} + A^{-1/4} \big)
\]
and hence that
\be\label{S1}
S_2 \ll N(\log^2 N) \big( N^{-\frac{1}{3\cdot 2^j}}+A^{-1/4} \big).
\ee
Bounding the inner sum over $c$ in  $S_2$ is exactly analogous,
where we use partial summation to remove the logarithm factor.
Since $N/b \ge N^{2/3}$, we obtain a stronger bound
\be\label{S2}
S_3\ll N(\log^2 N) \big( N^{-\frac{2}{3\cdot 2^j}}+A^{-1/4} \big).
\ee

For $S_4$, we break up the range $b \in (N^{1/3},2N^{2/3}]$ into $O(\log N)$
dyadic intervals of the form $(B,2B]$ where $N^{1/3}\le B \le 2N^{2/3}$.   Then we use Cauchy-Schwarz, followed by the trivial bound
$|h(b)| \le \tau(b)$ to get
\dalign{
S_4 &\ll (\log N) \max_B \Big| \sum_{B < b\le 2B} h(b)
\sum_{bc\in I} \Lambda(c) e(f(bc)) \Big| \\
&\le (\log N) \max_B \Big( \sum_{B<b\le 2B} h(b)^2 \Big)^{1/2}
\Big( \sum_{B<b\le 2B} \Big| \sum_{bc\in I} \Lambda(c) e(f(bc)) \Big|^2  \Big)^{1/2}\\
&\ll (\log N)^{5/2} \max_B B^{1/2} \Big( \sum_{B<b\le 2B} \Big| \sum_{bc\in I} \Lambda(c) e(f(bc)) \Big|^2  \Big)^{1/2}.
}
Next, we expand the square and then interchange the order of summation:
\be\label{sumbc1c2}
 \sum_{B<b\le 2B} \Big| \sum_{bc\in I} \Lambda(c) e(f(bc)) \Big|^2 
 =  \sum_{\frac{N}{2B} < c_1,c_2 \le \frac{2N}{B}} 
 \Lambda(c_1) \Lambda(c_2) \sum_{b\in J} e(f(bc_1)-f(bc_2)), 
\ee
where
\[
J = \{ B< n \le 2B : bc_1 \in I, bc_2 \in I \}
\]
is a subinterval of $(B,2B]$.
Let $R$ be a large constant, depending on $r_1,r_2$.  The terms above with $|c_1-c_2| \le \frac{RN}{BA^{1/5}}$ contribute
at most $O(N^2(\log N)^2/(A^{1/5} B))$ to the right side of \eqref{sumbc1c2}.
Now suppose that $|c_1-c_2| > \frac{RN}{BA^{1/5}}$.  Write
\[
f(bc_1)-f(bc_2) = \alpha b(c_1-c_2) + \sum_{r=r_1}^{r_2} \frac{\beta_r'}{b^r}, \qquad \beta_r' = \beta_r \Big( \frac{1}{c_1^r} - \frac{1}{c_2^r} \Big).
\]
We apply Lemma \ref{expsumf} with $\beta_r$ replaced by $\beta_r'$,
$N$ replaced by $B$, and $A$ replaced by
\[
A' = \frac{AN^{r_1} \beta_{r_1}'}{B^{r_1} \beta_{r_1}}.
\]
Since
\[
|\beta_r'| \order |\beta_r| \frac{|c_1-c_2|}{c_1^{r+1}},
\]
we see that
\begin{align*}
\left| \frac{\beta_{r}'}{\beta_{r_1}'} \right| &\ll N^{(r-r_1)/2} c_1^{-(r-r_1)} \\
&\ll N^{(r-r_1)/2} (N/B)^{-(r-r_1)} \\
&\ll B^{(r-r_1)/2} (N/B)^{-(r-r_1)/2} \\
&\ll B^{(r-r_1)/2} N^{-(r-r_1)/6},
\end{align*}
so that the hypotheses \eqref{betar} hold.   Also, $A' \ge A^{4/5}$ if $R$ is large enough, and therefore
\[
 \sum_{b\in J} e(f(bc_1)-f(bc_2)) \ll B \big( B^{-1/2^j}+ A^{-1/5}   \big).
\]
Summing over all pairs $c_1,c_2$ we see that the expression in
\eqref{sumbc1c2} is 
\[
\ll \frac{N^2}{B}(\log N)^2 (N^{-1/(3\cdot 2^j)}+A^{-1/5}),
\]
 and we conclude that
 \be\label{S3}
 S_4 \ll N (\log N)^4 \big( N^{-\frac{1}{3\cdot 2^j}} + A^{-1/10} \big).
 \ee
 Inserting \eqref{S1}, \eqref{S2} and \eqref{S3} into \eqref{sump-1}, this completes the proof.
\end{proof}

%
%
%
%
%
{\Large \section{Detecting fractional parts}
\label{sec:fracparts}}
%
%
%

In this section we apply harmonic analysis to detect
the simultaneous fractional parts of ratios of primes.
 Denote by  $\{x\}$ the fractional part of $x$.

We begin with a result of Selberg.

\begin{lem}\label{Selberg}
For any $K\in \NN$ and any non-empty interval $I\subset \RR/\ZZ$,
there is a trigonometric polynomial $S_{K,I}^+(x)=\sum_{|n|\le K} a_n e(nx)$
which majorizes the indicator function of $I$ and a trigonometric
polynomial $S_{K,I}^-(x)=\sum_{|n|\le K} b_n e(nx)$ which minorizes
the indicator function of $I$, and which satisfy the following:
\begin{itemize}
\item $\max(|a_n|,|b_n|) \le 4/(|n|+1)$ for all $n$.
\item $\int_{0}^1 S_{K,I}(x)^{\pm}\, dx = \text{length}(I) \pm \frac{1}{K+1}$.
\end{itemize}
\end{lem}

\begin{proof}
For details and explicit construction
of $S_{K,I}^\pm$, see Chapter 1 in \cite{HLM}, especially formulas (16)--(22).
\end{proof}

\noindent
\textbf{Definition.}
A subset $\cR$ of $\RR^k$ is said to be \emph{$t$-simple} if,
for any $1\le j\le k$ and any choice of $z_i\in \RR$
($i\ne j$), the 1-dimensional projection
 $\{ z_j : (z_1,\ldots,z_k)\in \cR \}$
consists of at most $t$ disjoint intervals.

\begin{prop}\label{fracparts}
Fix $\eps$, $\rho$ such that $0<\rho<\eps$ and let $k\in \NN$ with $\eps \le 1/k^2$.
Suppose that $1\le m\le x^{1/2}$, and $M_1,\ldots,M_k$ are integers such that
\begin{enumerate}
\item[(i)] $M_i \ge x^{\eps}$ for all i;
\item[(ii)] $x/2^k < M_1\cdots M_k m \le 2x$;
\item[(iii)] for all $i$, $M_i \not\in \bigcup_{s\le 1/\eps+1} (x^{(1-\rho)/s},4x^{1/s}]$.
\end{enumerate}
Let $\cR$ be any $t-$simple subset of 
\[
\{ (x_1,\ldots,x_k) : M_i<x_i \le 2M_i \; (1\le i\le k), x<mx_1\cdots x_k \le 2x \}.
\]
 and let $\cQ$ denote the set of all $k$-tuples $\qq=(q_1,\ldots,q_k)$ of primes such that $\qq\in \cR$.
  For each $1\le j\le k$,
let $s_{j}=\fl{\frac{\log x}{\log M_j}}-1$.
 Then, for some $\xi>0$, which depends only on $\eps$,$\rho$ and $k$, we have (writing $n=q_1\cdots q_km$)
\begin{align}
\#\Big\{\qq\in \cQ :  \forall j, q_j^{\ell} \Big| \bn \Big\} &= (1+O(k^2\eps)) \prod_{j=1}^k \bigg(1-2^{-s_j}\sum_{h=0}^{\ell-1} \binom{s_j}{h}  \bigg)  |\cQ| 
+O_{k,\eps} \pfrac{tx^{1-\xi}}{m}, \label{Qell} \\
  \# \Big\{\qq\in \cQ :  \forall j, q_j \nmid \bn \Big\}&= \frac{1+O(k^2\eps)}{2^{s_1+\cdots+s_k}} |\cQ| 
+O_{k,\eps} \pfrac{t x^{1-\xi}}{m}.\label{Q0}
\end{align}
\end{prop}

\begin{proof}
First, we make some preliminary observations concerning
the quantities $M_j$ and $q_j$.  
Let $1\le j\le k$.
By (ii) and (iii), $M_j \le x^{1-\rho}$, hence $s_j\ge 0$.
By definition,
\[
x^{\frac{1}{s_j+2}} < M_j \le x^{\frac{1}{s_j+1}}.
\]
However, (i) implies that $s_j\le 1/\eps-1$, and hence using (iii) we in fact have stronger inequalities for $M_j$, namely
\be\label{Mjsj}
4 x^{\frac{1}{s_j+2}} \le M_j \le x^{\frac{1-\rho}{s_j+1}} \qquad (1\le j\le k).
\ee
It will important for our argument below that small powers of the
primes $q_j$ stay away from $x$; the contrary case
when $q_j^b$ is close to $x$ for some small $b$ and some $j$, will be shown
to be very rare in the next section.


If $s_j=0$ for some $j$, then $M_j \ge 4 x^{1/2}$.
But $q_j>M_j$ and $q_j|n$ imply that
$q_j^2 > 8n$ and hence $q_j \nmid \bn$.  Thus, the inequalities
\eqref{Qell} and \eqref{Q0} follow trivially in this case.

Now assume that $s_j\ge 1$ for every $j$.
For each $\qq \in \cQ$, let $n=mq_1\cdots q_k$.
Since $M_j < q_j \le 2M_j$, \eqref{Mjsj} implies
that $n$ has exactly $s_j+2$ digits in base-$q_j$.
Moreover, the leading digit is much smaller than $q_j/2$
since by \eqref{Mjsj},
\[
\frac{n}{q_j^{s_j+2}} < \frac{2x}{M_j^{s_j+2}} \le \frac{2}{4^{s_j+2}} \le \frac{1}{32}.
\]
Hence there are $s_j$ base-$q_j$ digits which 
could possibly induce a carry when adding $n$ and $n$ in
base-$q_j$. 
Therefore, $\bn$ is divisible by $q_j^\ell$ if and only if 
for at least $\ell$ values of $s\in \{1,2,\ldots,s_j\}$
 we have $\{n/q_j^{s+1}\}>1/2$.  Likewise,
 $q_j \nmid \bn$ if and only if $\{n/q_j^{s+1}\}<1/2$
 for every $s$ in the range  $1\le s\le  s_j$.  

Now we return to the proof of  the Proposition.
The number of $\qq$ such that $q_i|m$ for some $i$ is
\[
\ll (k\log x) x^{1-\eps}/m,
\]
which is negligible and can be absorbed into the error 
terms in \eqref{Qell} and \eqref{Q0} if $\xi<\eps$.
For each $1\le j\le k$ and $1\le s\le s_j$, let $\sigma_{j,s} \in \{0,1\}$, and denote by $\Sigma$ the vector of the numbers
$\sigma_{j,s}$. 
For each $\Sigma$ let
\[
\cQ_\Sigma := \Bigg\{ \qq\in \cQ :  \Big\{ \frac{mq_1\cdots q_k}{q_j^{s+1}} \Big\} \in \Big[ \frac{\sigma_{j,s}}{2}, \frac{1+\sigma_{j,s}}{2} \Big) \; (1\le j\le k, 1\le s\le s_j)  \Bigg\}.
\]
Our main task is to prove that
\be\label{Qsigma}
|Q_\Sigma| = \frac{1+O(k^2\eps)}{2^{s_1+\cdots+s_k}} |\cQ| +
O_{k,\eps}\pfrac{tx^{1-\xi}}{m}.
\ee

By our earlier remarks, the left side of \eqref{Qell} is the sum of $\cQ_{\Sigma}$ over all $\Sigma$ such that $\sum_s \sigma_{j,s} \ge \ell$ for all $j$,
and the left side of \eqref{Q0} equals $\cQ_\Sigma$ for the single $\Sigma$
with $\sigma_{j,s}=0$ for all $j,s$.  Thus,
\eqref{Qell} and \eqref{Q0} follow from \eqref{Qsigma}.
 
In order to prove \eqref{Qsigma}, fix 
 $\Sigma$ and
apply Lemma \ref{Selberg} to the intervals
$[0,1/2]$ and $[1/2,1]$ and with 
\[
K = \fl{k\eps^{-2}}.
\]
Define 
\dalign{
\psi^{\pm}_{0,K}(x)=S_{K,[0,1/2]}^\pm(x) &= \sum_{|n|\le K} c_{0,n}^{\pm} e(nx), \\
\psi^{\pm}_{1,K}(x)=S_{K,[1/2,1]}^\pm(x) &= \sum_{|n|\le K} c_{1,n}^{\pm} e(nx).
}
 Then
\be\label{smoothapprox}
\sum_{\qq \in \cQ} \prod_{j=1}^k \prod_{s=1}^{s_j} \psi^-_{\sigma_{j,s},K}(mq_1\cdots q_k/q_j^{s+1}) \le |\cQ_\Sigma| \le 
\sum_{\qq \in \cQ} \prod_{j=1}^k \prod_{s=1}^{s_j} \psi^+_{\sigma_{j,s},K}(mq_1\cdots q_k/q_j^{s+1}).
\ee
Denote by $\blam$ an integral vector $(\lam_{j,s}:1\le j\le k, 1\le s\le s_j)$, 
where each component is bounded by $K$ in absolute value.
Focusing on the lower bound (the upper bound analysis is identical), we then have
\be\label{Dlower}
|\cQ_\Sigma| \ge \sum_{\qq \in \cQ} \ssum{\blam} \bigg( \prod_{j,s} c_{\sigma_{j,s},\lam_{j,s}}^{-} \bigg) e\bigg(
m \sum_{j,s} \lam_{j,s} \frac{q_1\cdots q_k}{q_j^{s+1}} \bigg).
\ee
Using Lemma \ref{Selberg}, we find that the main term ($\lam_{j,s}=0$ for every $j,s$) equals
\[
|\cQ| \prod_{j,s} \bigg( \int_0^1 \psi^{-}_{\sigma_{j,s},K}(u)\, du \bigg)^{} = \frac{|\cQ|}{2^{s_1+\cdots+s_k}} (1+O(1/K))^{s_1+\cdots+s_k} = \frac{1+O(k^2 \eps)}{2^{s_1+\cdots+s_k}} |\cQ|.
\]
Now $s_1+\cdots +s_k \ll k/\eps$ and recall that $\eps<1/k^2$.
 By Lemma \ref{Selberg}, $\sum_n |c_{\sigma,n}^\pm|\ll \log K$
and therefore we have
\be\label{CD1}
|\cQ_\Sigma| \ge (1+O(k^2\eps)) \frac{|\cQ|}{2^{s_1+\cdots+s_k}}+E,
\ee
where
\[
E \ll (O(\log K))^{O(k/\eps)} \max_{\blam\ne \mathbf{0}} \Bigg|
  \sum_{\qq\in \cQ} e\bigg(
m \sum_{j,s} \lam_{j,s} \frac{q_1\cdots q_k}{q_j^{s+1}} \bigg) \Bigg|.
\]
 Fixing $\blam\ne \mathbf{0}$, 
 let $h = \min\{j\le k: \lam_{j,s}\ne 0 \text{ for some }s\}$
and define $r = \min\{ s: \lam_{h,s}\ne 0\}$.
Fixing $q_i \; (i\ne h)$, the $t$-simplicity of $\cR$ implies that
the variable $q_h$ ranges over primes 
in at most $t$ subintervals $I$ (possibly $t=0$) of $(M_h,2M_h]$.
We have
\[
\sum_{j,s} \lam_{j,s} \frac{q_1\cdots q_km}{q_j^{s+1}} = 
 \a q_h 
+ \sum_{s=r}^{s_h} \lam_{h,s}\frac{P}{q_h^{s}} =: f(q_h).
\]
for some real number $\a$ (depending on $m$ and the $q_i$ for $i\ne h$)
and $P=(q_1\cdots q_k m )/q_h$.  By (ii) and (iii),
\be\label{Plwr}
P \ge \frac{M_1\cdots M_k m}{M_h} \ge \frac{x}{2^k M_h} \ge 
x^{\rho} 2^{-k} M_h^{s_h}.
\ee
We also have  $|\lam_{h,s}|\le K \ll M_h^{1/10}$ for large $x$.
 Therefore, for each interval $I$ we may apply
Lemma \ref{expsumprimes} with
\[
N=M_h, \quad r_1=r, \quad \b_{r_1} = P\lam_{h,r}, \qquad A =  2^{-k}x^{\rho}.
\]
The condition $|\beta_{r_1}| \ge N^{r_1}A$ follows from \eqref{Plwr},
and the lower bound $M_h\ge x^{\eps}$ implies that $A\le M_h$, so that \eqref{betar} holds.
We also have that 
\[
j \le 3 + \frac{\log (KP)}{\log M_h} \le 3+\frac{\log x}{\log M_h} \le 3+1/\eps.
\]
Therefore, applying Lemma \ref{expsumprimes}, we get
\[
\sum_{q_h\in I} e(f(q_h)) \ll_k M_h(\log M_h)^4 \( M_h^{-\frac{1}{3\cdot 2^j}}+ x^{-\rho/4}\) \ll x^{-\xi} M_h.
\]
Summing over all $q_i \; (i\ne h)$, we find that
$E \ll_{k,\eps} t x^{1-\xi}.$   Combined with \eqref{CD1}, this completes
the proof of \eqref{Qsigma}.
\end{proof}

%
%
%
{\Large \section{Proof of Theorem \ref{main_theorem}}\label{sec:thm1}}
%

Throughout this section, we will assume that $k$
is a large integer, and that $\eps,\delta$ are
functions of $k$ that tend to 0 as $k\to \infty$; precisely, we take
\be\label{delepsrho}
\delta= e^{-2k/3}, \qquad \eps=k^{-2k}.
\ee
Suppose that $x$ is a large integer.  We think of $k$ being fixed and $x\to \infty$.
In this section only, we adopt the following notation for functions $f(k,x)$.
The notation $f(k,x)=o(g(k,x))$ means that 
\[
\forall k\ge 1 : \, \lim_{x\to \infty} \frac{f(k,x)}{g(k,x)} = 0.
\]
\newcommand{\ob}{\overline{o}}
The notation  $f(k,x)=\ob(g(k,x))$ means that
\[
\lim_{k\to \infty} \limsup_{x\to\infty} \frac{f(x,k)}{g(x,k)}=0.
\]
For example, $1/k = \ob(1)$ and $e^k x^{1-1/k} = o(x)$.

\subsection{Sampling large prime factors}\label{sec:model}
Take a large integer $x$, and select a random integer $n\in (x,2x]$
with uniform probability.   Following Donnelly and Grimmett \cite{DG}, we select at random a $k$-tuple 
$\qq(n)=(q_1,\ldots,q_k)$ of prime power divisors of $n$ at random, in a size-biased fashion, together
with random variables $X_1(n),\ldots,X_k(n)$.
If $n$ has fewer than $k$ distinct prime factors, set
$\qq(n)=(1,\ldots,1)$ and $X_1(n)=\cdots=X_k(n)=0$.  Otherwise,
choose $q_1|n$ at random with probability $\frac{\Lambda(q_1)}{\log n}$, where $\Lambda$ is the von Mangoldt function.  For
$2\le i\le k$, once $q_1,\ldots,q_{i-1}$ are chosen, select
$q_i|(n/q_1\cdots q_{i-1})$ with probability $\frac{\Lambda(q_i)}{\log (n/q_1\cdots q_{i-1})}$.
Then set 
\[
X_i(n) = \frac{\log q_i}{\log (n/q_1\cdots q_{i-1})} \qquad (1\le i\le k)
\]
We observe the relation
\be\label{qiX}
q_i = n^{(1-X_1(n))\cdots (1-X_{i-1}(n))X_i(n)} \qquad (1\le i\le k).
\ee

The following is essentially Theorem 1 of \cite{DG}, although we
have stated the result with a slight modification.  For completeness, a proof is given in the Appendix.

\medskip

\begin{lem}\label{DG}
Fix $k\in \NN$. 
As $x\to \infty$, the random vector $(X_1(n),\ldots,X_k(n))$ 
converges weakly to the uniform distribution (that is, Lebesgue measure) on $[0,1]^k$.
\end{lem}

We denote $\PR_x$, $\E_x$ for the probability, respectively expectation, with respect to
these random $n$, $\qq(n)$ and $(X_1(n),\ldots,X_k(n))$,
and use $\PR$ and $\E$ for the uniform probability measure on $[0,1]^k$.
For the latter, we work with  independent, uniform-$[0,1]$ random variables $U_1,\ldots,U_k$.

\medskip

\noindent
\textbf{Definition.}
With $x$ fixed, let $\cY_k(x)$ denote the set of $k$-tuples $\mathbf{y}=(y_1,\ldots,y_k)\in [1,x]^k$ such that
\begin{enumerate}
\item[(a)] $y_i \ge  x^{\eps}$ for all $i$;
\item[(b)] $x^{1-\del} \le y_1\cdots y_k \le x^{1-\del^2}$;
\item[(c)] for all $i$ and all $1\le s\le 1/\eps+1$,
$y_i \not\in [x^{(1-\eps^2)/s},8x^{1/s}]$.
\end{enumerate}

\begin{lem}\label{Yksimple}
The set $\cY_k(x)$ is $(1/\eps+2)$-simple.
\end{lem}

\begin{proof}
Fix $j$ and let $y_i$ be arbitrary for $i\ne j$.
Items (a) and (b) force $y_j$ into a single interval, from which are cut at most $1/\eps+1$ intervals by (c).
\end{proof}

\begin{lem}\label{R}
We have $\PR_x(\qq(n) \not\in \cY_k(x) \text{ or some } q_i \text{ not prime}) = \ob(1)$.
\end{lem}

\begin{proof}
First, note that $\PR_x(n \text{ has fewer than }k\text{ prime factors})=o(1)$.  Now assume that $n$ has at least $k$ distinct
prime factors.  Write $q_i=q_i(n)$ for brevity.
 By \eqref{qiX} and Lemma \ref{DG},
\dalign{
\PR_x (\text{some } q_i<x^{\eps})
 &\le \PR_x (\text{some } q_i \le n^{\eps}) \\
 &\le \PR \big( (1-U_1)\cdots(1-U_{i-1})U_i \le \eps \text{ for some }i \big)+o(1) \\
&\le \PR \big( \, U_i \not\in [\eps^{1/k},1-\eps^{1/k}]\text{ for some }i  \big) +o(1) \\
&\le 2k\eps^{1/k} + o(1) = \ob(1),
}
upon recalling \eqref{delepsrho}. 

From \eqref{qiX}, we have 
\[
q_1 \cdots q_k = n^{1-(1-X_1(n))\cdots (1-X_k(n))}.
\]
Hence,
\dalign{
\PR_x \big( x^{1-\delta} \le q_1\cdots q_k \le x^{1-\delta^2} \big) &=
\PR_x \Bigg( \frac{\log n}{\log x} \big(1-(1-X_1(n))\cdots (1-X_k(n))\big) \in [1-\del, 1-\del^2] \Big).
}
By Lemma \ref{DG},  as $k\to \infty$, the variable
$1-(1-X_1(n))\cdots (1-X_k(n))$ converges in distribution to
$1-(1-U_1)\cdots (1-U_k)$.
Now $\E \log(1-U_i) = -1$ for each $i$, and it follows from the Law of Large Numbers that 
\be\label{LLN}
\PR \big( (1-U_1)\cdots (1-U_k) \in [e^{-1.1k},e^{-0.9k}] \big)
=1-\ob(1).
\ee
Recalling the definition of $\delta$ from \eqref{delepsrho},
we conclude that
\[
\PR_x \big(q_1\cdots q_k\not \in [x^{1-\del},x^{1-\del^2}]\big) = \ob(1).
\]

The probability that (c) fails is at most the probability that $n$ has a prime power factor in one
 of the intervals $[x^{(1-\eps^2)/s},8x^{1/s}]$,
which is easily bounded by Mertens' theorem by
\[
 \sum_{s\le 1/\eps+1} \;\; \sum_{x^{(1-\eps^2)/s} < q \le 8x^{1/s}} \frac{1}{q} \ll  \frac{\eps^2}{\eps} =\eps  = \ob(1).
\]

Finally, if every $q_i \ge x^{\eps}$ and some $q_i$ is not prime, then 
$n$ is divisible by a prime power $p^a>x^\eps$ with $a\ge 2$.
The number of such $n\in (x,2x]$ is $O(x^{1-\eps/2})$.
This completes the proof.
\end{proof}

\subsection{Completing the proof.}
From now on, the variables $q_i$ will denote primes.
Let $n$ and $\qq(n)$ be the random quantities described above.   Our main task is to show that
\be\label{goal}
\PR_x \( n^\ell \Big| \bn \) = c_\ell + \ob(1).
\ee
Theorem \ref{main_theorem} follows immediately upon fixing $k$,
letting $x\to \infty$, and then letting $k\to \infty$.

We first show, using Proposition \ref{smallprfact} and Lemma \ref{R}
that it suffice to consider large prime factors of $n$ and
$\qq(n) \in \cY_k(x)$. 
Let
\[
B_n = \sprod{p^a \| n \\ p>y} p^a,
\]
where $y$ is the smallest power of two that is $>x^{2\delta}$.
Applying Proposition \ref{smallprfact}, followed by an application
of Lemma \ref{R}, we see that
\be\label{Px1}
\PR_x \( n^\ell \Big| \bn \) = \ob(1) + \PR_x \( B_n^\ell \Big| \bn \) = \ob(1) + \PR_x \(   B_n^\ell \Big| \bn \text{ and } \qq(n)\in \cY_k(x)  \).
\ee
If $\qq(n) \in \cY_k(x)$, then
 by (b), $q_1\cdots q_k \ge x^{1-\del}$.
It follows that $B_n | q_1 \cdots q_k$, that is, $q_1\cdots q_k$
contains all of the large prime factors of $n$. 
On the other hand, Proposition \ref{smallprfact} implies that
the probability that some prime factor $q<y$ of $n$ satisfies $q^\ell \nmid \bn$
is $\ob(1).$ Thus
\[
\PR_x  \( B_n^\ell \Big| \bn \text{ and } \qq(n)\in \cY_k(x)  \)  = \PR_x \( \qq(n) \in \cY_k(x) \wedge q_j^\ell \Big| \bn \; (1\le j\le k) \)+\ob(1).
\]

Combined with \eqref{Px1}, this gives
\be\label{Px2}
\PR_x \( n^\ell \Big| \bn \) = \ob(1) + \sum_{\qq\in \cY_k(x)}
\PR_x \( \qq(n)=\qq  \wedge q_j^\ell \Big| \bn \; (1\le j\le k) \).
\ee
Write $n=mq_1\cdots q_k$.
Direct computation gives
\dalign{
\PR_x \( \qq(n)=\qq  \wedge q_j^\ell \Big| \bn \; (1\le j\le k) \)
&= \frac{1}{x}
\ssum{ x < mq_1\cdots q_k \le 2x\\ q_j^\ell | \bn \, (1\le j\le k)} \frac{(\log q_1)\cdots (\log q_k)}{\log n \log (n/q_1) \cdots \log n/(q_1\cdots q_{k-1})}.
}
It is convenient to place each $q_i$ into
a dyadic interval.  For each $i$, let 
$M_i$ be the unique power of two such that $M_i < q_i \le 2M_i$.
By conditions (b) and (c) in the definition of $\cY_k(x)$,
\be\label{qM}
\frac{(\log q_1)\cdots (\log q_k)}{\log n \log (n/q_1) \cdots \log n/(q_1\cdots q_{k-1})}
= (1+o(1)) \frac{(\log M_1)\cdots (\log M_k)}{\log x \log (\frac{x}{M_1}) \cdots \log (\frac{x}{M_1\cdots M_{k-1}})} .
\ee
We insert this last estimate into \eqref{Px2}, obtaining
\begin{align}
\PR_x \( n^\ell \Big| \bn \) &= \ob(1) + (1+o(1)) \sum_{\MM} \frac{(\log M_1)\cdots (\log M_k)}{\log x \log \pfrac{x}{M_1} \cdots \log \pfrac{x}{M_1\cdots M_{k-1}}} \notag 
\\
&\qquad \times \sum_{\frac{x}{2^k M_1 \cdots M_k} < m \le \frac{2x}{M_1 \cdots M_k}} \;\;
\ssum{ \qq\in \cR(\MM,n) \\ q_j^\ell | \bn \, (1\le j\le k)} 1, \label{Px3}
\end{align}
where the sum is taken over $\MM = (M_1,\dots,M_k)$ with each $M_i$ a power of two, and
we have written $n=q_1\cdots q_k m$ and
\[
\cR(\MM,m) = \{ (z_1,\ldots,z_k)\in \cY_k(x) : 
M_i <z_i \le 2M_i \; (1\le i\le k), x < mz_1\cdots z_k \le 2x \}.
\]
Now fix $\MM$ and $m$.
By Lemma \ref{Yksimple}, $\cY_k(x)$ is $(1/\eps+2)$-simple
and thus $\cR(\MM,m)$ is also $(1/\eps+2)$-simple.
 We may then apply Proposition \ref{fracparts} to $\cR(\MM,m)$.
Condition (iii) in that Proposition holds with $\rho=\eps^2$
on account of (c). Indeed, if 
$$M_i\in\left(x^{(1-\rho)/s},4x^{1/s}\right),$$
then
$$q_i\in\left(x^{(1-\rho)/s},8x^{1/s}\right),$$
and (c) does not hold.
Let $s_j=\lfloor \frac{\log x}{\log M_j}\rfloor-1$ for each $j$, 
and define
\[
F(b)= 1- 2^{-b}\sum_{h=0}^{\ell-1} \binom{b}{h} ,
\]
 By Proposition \ref{fracparts}, we get that
 \[
\ssum{ \qq\in \cR(\MM,m) \\ q_j^\ell | \bn \, (1\le j\le k)} 1
  = 
(1+O(k^2 \eps)) \prod_{j=1}^k F(s_j)  
\ssum{ \qq\in \cR(\MM,m)} 1 +O_{k,\eps}(x^{1-\xi}),
 \]
 for some $\xi>0$.  The final error term is negligible since
 the number of $\MM$ is $\ll_k (\log x)^k$. 
Now sum over all $m$ and $\MM$, and rewrite the final
result in terms of $\qq$ using \eqref{qM} again.
   By \eqref{Px3} and $O(k^2 \eps)=\ob(1)$
 we conclude that
 \begin{align}
\PR_x \( n^\ell \Big| \bn \) &= \ob(1) + (1+\ob(1)) \sum_{\qq\in \cY_k(x)} \PR_x( \qq(n) = \qq) \prod_{j=1}^k F(s_j) \notag \\
&=  \ob(1) + (1+\ob(1))\E_x
\mathbf{1}_{\qq(n) \in \cY_k(x)} \prod_{j=1}^k F(s_j),
 \label{Px4}
\end{align}
where (consistent with the earlier definition) by (c) we have for large enough $x$
\be\label{sjM}
s_j = \fl{\frac{\log x}{\log q_j}}-1  \qquad (1\le j\le k, \qq\in \cR(\MM,m)).
\ee
Indeed, clearly,
\[
\fl{\frac{\log x}{\log q_j}}\le \fl{\frac{\log x}{\log M_j}},
\]
and it suffices to show that
\[
\fl{\frac{\log x}{\log q_j}}\ge s_j+1.
\]
We have $M_j\le x^{1/(s_j+1)}$, and next, by (c),
$M_j\le x^{(1-\rho)/(s_j+1)}$. Hence,
$q_j\le 2x^{(1-\rho)/(s_j+1)}\le x^{1/(s_j+1)}$, as required
for \eqref{sjM}.
 
Using Lemma \ref{R} again, followed by Lemma \ref{DG},
we arrive at
\[
\PR\Big( n^\ell \Big| \bn \Big) = \ob(1) + \E_x \prod_{j=1}^k F(s_j) =  \ob(1) + \E \prod_{j=1}^k F(g_j),
\]
where $g_j$ is defined in \eqref{kj}.
Finally, by the Law of Large Numbers, cf. \eqref{LLN} 
we have $g_j \ge e^{j/2}$ for all $j\ge k$ with probability
$1-\ob(1)$ and this completes the proof of \eqref{goal}
upon recalling that
\[
c_{\ell} = \E \prod_{j=1}^\infty F(g_j).
\hfill \hfill
\]


{\Large \section{Proof of Theorem \ref{thm:coprime}}\label{sec:coprime}}

The proof is similar to that of Theorem \ref{main_theorem},
but the details are simpler.  In particular, we do not need the
work from Section \ref{sec:model}.  As before, the symbols
$q$ and $q_i$ denote primes.

For fixed $k\in \NN$ and $\eps>0$ let
\[
\cN_{k,\eps}(x) = \# \Big\{ n=q_1\cdots q_k \in (x,2x] : \(n,\bn\)=1,  \forall i,
\; q_i\ge x^{\eps} \text{ and } q_i\not\in \bigcup_{s\le 1/\eps+1}
(x^{(1-\eps^3)/s},8x^{1/s}] \Big\}.
\] 
In contrast to the argument of the previous section, here
we will take $\rho=\eps^3$, for reasons that will become
apparent later.

\begin{lem}\label{Nkx}
For any fixed $k\ge 2$ and $\eps>0$ we have
\[
|\cN_{k,\eps}(x)| = \frac{x}{\log x} 
\Bigg\{ \frac{1}{k!}\;\; \idotsint\limits_
{\substack{\mathbf{u}\in[\eps,1]^k \\ u_1+\cdots+u_k=1}} 
h(u_1) \cdots h(u_k)\, du_1 \cdots du_{k-1} + O_k(\eps^2) +O_{k,\eps} \pfrac{1}{\log x} \Bigg\},
\]
where $h(v) = v^{-1} 2^{1-\fl{1/v}}$.
\end{lem}
 
\begin{proof}
Consider $n\in \cN_{k,\eps}(x)$, and write $n=q_1\cdots q_k$
with $q_1<\cdots < q_k$.  Let
\[
\cT = \Bigg\{ x^{\eps} \le y_1 < \cdots < y_k \le x : x<y_1\cdots y_k \le 2x, \forall i:\, y_i\not \in \bigcup_{s\le 1/\eps+1}
\left(x^{(1-\eps^3)/s},8x^{1/s}\right] \Bigg\},
\]
so that $\qq=(q_1,\ldots,q_k)\in \cT$.
For each $i$, let $M_i$ be the unique power of two such that
$M_i < q_i \le 2M_i$, and for a fixed $\MM=(M_1,\ldots,M_k)$
let $T(\MM) = \{ \mathbf{y}\in \cT : M_i<y_i \le 2M_i\; (1\le i\le k)\}$.

With $\MM$ fixed, define $s_j=\lfloor \frac{\log x}{\log M_j}\rfloor -1$. 
Then the hypotheses of Proposition \ref{fracparts} hold
with $\rho=\eps^3$.
 The set $\cT$ is $(1/\eps+2)-$simple and hence
 by Proposition \ref{fracparts} with $m=1$, we get that
 \be\label{Nkex}
 |\cN_{k,\eps}(x)| = 
\ssum{ \qq\in \cT(\MM) \\ \big(q_1\cdots q_k,\bn\big)=1} 1
  = 
(1+O(k^2 \eps)) 2^{-(s_1+\cdots+s_k)}
\ssum{ \qq\in \cT(\MM)} 1 +O_{k,\eps}(x^{1-\xi}).
 \ee
Using that $\cT$ is $(1/\eps+2)$-simple,
repeated application of the prime number theorem
with classical error term implies that, for some
fixed positive $c$,
\begin{align*}
\ssum{ \qq\in \cT(\MM)} 1 &= \int_{\cT(\MM)} \frac{d\mathbf{y}}{(\log y_1)\cdots (\log y_k)} + 
O_{k,\eps}( M_1 \cdots M_k e^{-c \min_i \sqrt{\log M_i}}) \\
&=\int_{\cT(\MM)} \frac{d\mathbf{y}}{(\log y_1)\cdots (\log y_k)} + O_{k,\eps}( x e^{-c \sqrt{\eps \log x}}).
\end{align*} 
Now for any $\mathbf{y}\in \cT(\MM)$, due to the arguments used in the previous section, we have 
$s_j=\fl{\frac{\log x}{\log y_j}}-1$ for each $j$. 
There are $\ll_k (\log x)^k$ possible tuples $\MM$. 
Thus, after summing over all $\MM$
and recalling \eqref{Nkex},
 we obtain
\begin{equation}
\label{sum_M}
 |\cN_{k,\eps}(x)| = O_{k,\eps}(x^{1-\xi/2} + x/\log^5 x) + 
 (1+O_k(\eps)) 
 \int_{\cT} \prod_{j=1}^k \frac{2^{1-\lfloor \frac{\log x}{\log y_j} \rfloor}}{\log y_j}
 d\mathbf{y}.
 \end{equation}
 
Making the change of variables $u_i = \frac{\log y_i}{\log x}$
for each $i$, and recalling the definition of $h(\cdot)$, we see that
\[
\int_{\cT} \prod_{j=1}^k \frac{2^{1-\lfloor \frac{\log x}{\log y_j} \rfloor}}{\log y_j} d\mathbf{y} = \int_{\cU}
h(u_1)\cdots h(u_k) x^{u_1+\cdots +u_k} \,
du_1 \cdots du_k, 
\]
where
\[
\cU = \Big\{ \eps \le u_1 \le \cdots \le u_k\le 1 : 1\le u_1+\cdots+
u_k\le 1+\frac{\log 2}{\log x}; \forall i, u_i\not
 \in \bigcup_{s\le 1/\eps+1}\Big[ \frac{1-\eps^3}{s},\frac{1}{s} +\frac{\log 8}{\log x}\Big] \Big\},
\] 
Replacing the condition $\eps \le u_1 \le \cdots \le u_k\le 1$ with the condition $\mathbf{u} \in [\eps,1]^k$
introduces a factor $1/k!$ in the integral, as the remaining conditions in
the definition of $\cU$ are symmetric in the variables
$u_1,\ldots,u_k$.  In addition, the set of $\mathbf{u} \in [\eps,1]^k$
that satisfy $1\le u_1+\cdots+u_k \le 1+\frac{\log 2}{\log x}$
and also
 $u_i \in [\frac{1-\eps^3}{s},\frac{1}{s}+\frac{\log 8}{\log x}]$ for some $i\le k$ and some $s\le 1/\eps+1$
has Lebesgue measure $O(k\eps^2/\log x)$.
The integrand is $O(2^k x)$ and therefore
\[
\int_{\cT} \prod_{j=1}^k \frac{2^{1-\lfloor \frac{\log x}{\log y_j} \rfloor}}{\log y_j} d\mathbf{y}
 = \frac{1}{k!} \int_{\cV} x^{u_1+\cdots + u_k} h(u_1) \cdots h(u_k) du_1\cdots du_{k}
 + O \pfrac{\eps^2 x}{\log x},
\]
where
\[
\cV = \bigg\{\mathbf{u}\in [\eps,1]^k  :1\le u_1+\cdots+u_k\le 1 + \frac{\log 2}{\log x} \bigg\}
\]

Notice that in the region $\cV$,
$u_i \le 1-\eps/2$ for all $i$ (assuming $x\ge \exp(10/\eps)$, say).
Further analysis is complicated by the discontinuities of
$h(u)$ at $u=1/s$, $s\in \NN$.  The function $h()$ is,
however, bounded by 2.  We'll replace the function
$h$  by the continuous function
$h_\eps(u)$ on $0\le u\le 1$, which equals $h(u)$ whenever
$|u-1/s| \ge \eps^4$ for all $2\le s\le 1/\eps+1$,
and otherwise is linear on each segment 
$[1/s-\eps^4,1/s+\eps^4]$, $2\le s\le 1/\eps+1$.
As before, the set of $\mathbf{u}\in \cV$ that
also satisfy $|u_i-1/s| \ge \eps^4$ for $i$ and
 some $2\le s\le 1/\eps+1$
has Lebesgue measure $O(k\eps^3/\log x)$.  We thus obtain
\be\label{TV}
\int_{\cT} \prod_{j=1}^k \frac{2^{1-\lfloor \frac{\log x}{\log y_j} \rfloor}}{\log y_j}  d\mathbf{y}
 = \frac{1}{k!} \int_{\cV} x^{u_1+\cdots + u_k} h_\eps(u_1) \cdots h_\eps(u_k) du_1\cdots du_{k}
 + O \pfrac{\eps^2 x}{\log x}.
\ee
Since $h(u)$ has bounded derivative on $[0,1) \setminus \{1/2,1/3,1/4,\ldots\}$,
the function $h_\eps$ satisfies 
\[
|h_\eps(a)-h_\eps(b)| \ll \eps^{-4} |a-b| \qquad (a,b\in [0,1]).  \]
Hence,
letting $v=u_1+\cdots+u_k$, and using that
$|u_i - u_i/v| \ll 1/\log x$ for each $i$, we get
\begin{align*}
\int_{\cV} x^{u_1+\cdots + u_k} h_\eps(u_1)& \cdots h_\eps(u_k) du_1\cdots du_{k} =  \int_{\cV} x^v h_\eps(u_1/v) \cdots h_\eps(u_k/v) \, du_1\cdots du_k +O_{k,\eps}\pfrac{x}{\log^2 x}\\
&= \int_1^{1+\frac{\log 2}{\log x}} x^v
\idotsint\limits_{\substack{\mathbf{u}\in [\eps,1]^k \\ u_1+\cdots+u_k=v}} h_\eps(u_1/v)\cdots h_\eps(u_k/v)\, du_1\cdots du_{k-1}\, dv+O_{k,\eps}\pfrac{x}{\log^2 x}\\
&=\int_1^{1+\frac{\log 2}{\log x}} x^v v^{k-1} \, dv
\idotsint\limits_{\substack{\mathbf{u}\in [\eps,1]^k \\ u_1+\cdots+u_k=1}} h_\eps(u_1)\cdots h_\eps(u_k)\, du_1\cdots du_{k-1}+O_{k,\eps}\pfrac{x}{\log^2 x}.
\end{align*}
Now $v^{k-1} = 1+O_k(1/\log x)$.  Recalling \eqref{TV},
we arrive at
\[
\int_{\cT} \prod_{j=1}^k \frac{2^{1-\lfloor \frac{\log x}{\log y_j} \rfloor}}{\log y_j} d\mathbf{y}= \frac{x}{k! \log x}
\idotsint\limits_{\substack{\mathbf{u}\in [\eps,1]^k \\ u_1+\cdots+u_k=1}} h_\eps(u_1)\cdots h_\eps(u_k)\, du_1\cdots du_{k-1}+O_{k,\eps}\pfrac{x}{\log^2 x}+O\pfrac{\eps^2 x}{\log x}.
\]
We conclude by replacing each
$h_\eps(u_i)$ with $h(u_i)$.
Since the set
\[
\{ \mathbf{u}\in [\eps,1]^k : u_1+\cdots+u_k=1; \exists i,
h(u_i)\ne h_\eps(u_i) \}
\]
has $(k-1)$-dimensional Lebesgue measure $O(k \eps^3)$,
this produces an additive error term
of order $O(\eps^3 x/\log x)$
(again, using that $h()$ and $h_{\eps}$
are bounded).  Thus, recalling \eqref{Nkex}
and \eqref{sum_M}, the proof is complete.
\end{proof}

 \begin{proof}[Proof of Theorem \ref{thm:coprime} from Lemma \ref{Nkx}]
 Let $\cN_k$ be the set of $n\in (x,2x]$ with $k$ distinct 
prime factors and with $(n,\bn)=1$.
Fix $\eps>0$.   Clearly 
\[
\cN_1 \sim \frac{x}{\log x}.
\]
Now let $k\ge 2$.
  Then one of the following is true for any
$n\in \cN_k$:
\begin{enumerate}
\item[(1)] $n\in \cN_{k,\eps}(x)$;
\item[(2)] $n$ has a prime factor smaller than $x^{\eps}$;
\item[(3)] $n$ is divisible by the square of some prime larger than $x^{\eps}$; or
\item[(4)] $n$ has a prime factor in $\bigcup_{s\le 1/\eps+1}
(x^{(1-\eps^3)/s},4x^{1/s}]$. 
\end{enumerate}
Lemma \ref{Nkx} gives the size of $\cN_{k,\eps}(x)$.
By Proposition \ref{coprime-smpr}, the number of
$n$ satisfying (2) is $O(e^{-1/(3\eps)} x/\log x)$.
The number of $n$ satisfying (3) is evidently 
$\ll x^{1-\eps/2}$.  
Fixing $s$, the number of $n\in \cN_k$,
with all prime factors $\ge x^{\eps}$ and with a prime factor in
$I=(x^{(1-\eps^3)/s},4x^{1/s}]$ is 
zero for $s=1$, and when $s\ge 2$ it is
at most
\dalign{
\sum_{q_1\in I}
\ssum{q_2,\cdots,q_{k-1} \\ \forall i: \, q_i\ge x^{\eps} \\
 q_1\cdots q_{k-1} \le 2x^{1-\eps}} \pi\pfrac{x}{q_1\cdots q_{k-1}} &\ll \sum_{q_1\in I} \sum_{q_2,\ldots,q_k\in (x^{\eps},x]}
 \frac{x}{\eps q_1\cdots q_{k-1} \log x} \\
 &\ll \frac{x}{\log x} \, \frac{(\log 2/\eps)^{k-1} \eps^3}{\eps}.
}
After summing the above over $s\le 1/\eps+1$, we
see that the number of $n$ satisfying (4) is
\[
\ll \frac{\eps(\log (2/\eps))^{k-1} x}{\log x}.
\]
We conclude that
\[
|\cN_k| = \frac{x}{\log x} \Bigg\{ \frac{1}{k!} \;\; \idotsint\limits_{\substack{\mathbf{u}\in [\eps,1]^k \\ u_1+\cdots+u_k=1}}  
h(u_1) \cdots h(u_k)\, du_1\cdots du_{k-1} + O\( e^{-1/(3\eps)} 
+ \eps (\log 2/\eps)^{k-1} + o(1) \)  \Bigg\}.
\]
The function $h()$ is bounded above by 2, thus upon letting
$\eps\to 0$ we find that
\be\label{Nk}
|\cN_k| \sim \frac{x}{k! \log x} \idotsint
\limits_{\substack{0\le u_1,\ldots,u_k\le 1 \\ u_1+\cdots + u_k=1}} 
h(u_1) \cdots h(u_k)\, du_1\cdots du_{k-1}
\qquad (x\to \infty)
\ee
for each fixed $k$.  On the other hand, if $n$ has more than $K$
prime factors, then $n$ has  a prime factor $<x^{1/K}$,
and by Proposition \ref{coprime-smpr}, there are $O(e^{-K/3} x/\log x)$ such integers.  That is, for any fixed $K$,
\[
\# \{ \cB \cap [1,x] \} = \sum_{k=1}^K |\cN_k| + O\( e^{-K/3} \frac{x}{\log x} \).
\]
Again using that $h(u) \le 2$ for all $u$, we wee that
$|\cN_k| \le \frac{2^k}{(k!)^2}\frac{x}{\log x}$.
Thus, letting $K\to \infty$, Theorem \ref{thm:coprime} follows.
 \end{proof}


{\Large \section{Numerical estimates of the density}\label{sec:numerics}}

It is convenient here to go back to the variables $Y_i$
 given in \eqref{Un}.
Moreover, in order for the product in the definition to be nonzero, 
we need $Y_i \le \frac{1}{\ell+1}$ for all $i$.  In particular, this shows that
\be\label{clupp}
c_\ell \le \rho(\ell+1) = e^{-(1+o(1))\ell \log \ell}
\ee
as $\ell\to \infty$,
where $\rho$ is the Dickman function.  We have
\be\label{clprod}
c_\ell = \E \prod_{j=1}^\infty g(Y_j), \quad g(y)=
\begin{cases} 1-2^{1-\fl{1/y}}
\sum_{h=0}^{\ell-1} \binom{\fl{1/y}-1}{h} & \text{ if } 0<y\le \frac{1}{\ell+1} \\ 0 & \text{ if } y>\frac{1}{\ell+1}.
\end{cases}
\ee
We estimate $c_\ell$ using Laplace transforms.  By
Theorem 3.2 of \cite{Handa}, we have that
\be\label{laplace-F}
F(s) := \int_0^\infty e^{-st} \Big(\E \prod_{j=1}^\infty g(t Y_j) \Big)\, dt = \frac{1}{s}
\exp \Bigg( \int_0^\infty \frac{g(z)-1}{z} e^{-sz}\, dz  \Bigg)
\qquad (\Re s>0) .
\ee
Theorem 3.2 of \cite{Handa} is only stated for real $s>0$, but
the proof gives the result in the full half-plane $\Re s>0$.
The left side of \eqref{laplace-F} is an entire function of $s\in \CC$, since
\[
\E \prod_{j=1}^\infty g(t Y_j) \le \rho(t(\ell+1))
\]
decays faster than exponentially in $t$; however the right side 
is only well defined for $\Re s>0$.
We massage the right side using the standard function
\be\label{E1}
E_1(z) = \int_z^\infty \frac{e^{-t}}{t}\, dt.
\ee
Since 
$g(z)=0$ for $z>\frac{1}{\ell+1}$ we may decompose
\[
\int_0^\infty \frac{g(z)-1}{z} e^{-sz}\, dz =
\int_0^{\frac{1}{\ell+1}} \frac{g(z)-1}{z} e^{-sz}\, dz
- E_1\pfrac{s}{\ell+1}.
\]
We next use the fact that $g(z)$ is a
step-function with jumps at the points $1/k$, where
$k$ is an integer satisfying $k\ge \ell+1$.
Using the Pascal relation, and in the notation of
Stieltjes integration, we have
\dalign{
dg\pfrac{1}{k} = g\pfrac{1}{k-1} - g\pfrac{1}{k} 
&= -2^{2-k}\sum_{h=0}^{\ell-1} \binom{k-2}{h} +
2^{1-k} \sum_{h=0}^{\ell-1} \( \binom{k-2}{h-1}+\binom{k-2}{h}\)\\
& = -2^{1-k} \binom{k-2}{\ell-1}. 
}
Thus, applying (Stieltjes) integration by parts we find that
\dalign{
  \int_0^{(1/(\ell+1))^+} (g(z)-1) \frac{e^{-s z}}{z}\, dz &=  E_1\pfrac{s}{\ell+1}
  + \int_0^{(1/(\ell+1))^+} E_1 (s z) d g(z) \\
&=  E_1\pfrac{s}{\ell+1} - \sum_{k\ge \ell+1} 2^{1-k}
\binom{k-2}{\ell-1} E_1\pfrac{s}{k}.
}
Here we used that $\lim_{y\to 0^+} g(y)=1$ and
$\lim_{z\to 0} E_1(sz)(g(z)-1)=0$.
Inserting this into \eqref{laplace-F} and inverting,
we conclude the following:

%
%
%
%
\begin{prop}\label{clJ}
For  any $\sigma>0$, we have
\[
c_\ell = \frac{1}{2\pi i} \int_{\sigma-i\infty}^{\sigma+i\infty}
\frac{e^{s}}{s} \exp\Big\{ - \sum_{k\ge \ell+1} 2^{1-k}
\binom{k-2}{\ell-1} E_1\pfrac{s}{k}  \Big\} \, ds.
\]
\end{prop}

Computing $c_\ell$ was accomplished with the Python scripts \texttt{mpmath}, which have a built-in
function for numerically inverting the Laplace transform, and 
which can can be computed to arbitrary precision.
Table \ref{table:c_ell} shows truncated values 
with precision 50, 100 and 200 digits.
The values for $\ell=1$ are unstable in the 8th decimal place,
while the calculations appear more accurate for larger $\ell$.

\begin{table}[h]
\begin{lstlisting}
from mpmath import *  
mp.dps=100  # digit accuracy of internal computations
def F(s,l):
    x=mpf('0.0')
    for k in range(l+1,200):x=x+2**(1-k)*binomial(k-2,l-1)*mp.e1(s/k)
    return(mp.exp(-x)/s)
c = lambda l : mp.invertlaplace(lambda z: F(z,l),1)
\end{lstlisting}
\caption{Python code to compute $c_\ell$}
\end{table}

\begin{table}[h]
\begin{tabular}{|c|r|r|r|c|}
\hline
$\ell$ & mp.dps=50 & mp.dps=100 & mp.dps=200 & scale\\
\hline
1 & 0.114247499194 & 0.114247430441 & 0.114247438905 & 1 \\
2 & 3.227780974290 & 3.227778322653 & 3.227778439553 & $10^{-3}$\\
3 & 3.151177764641 & 3.151177748965 & 3.151177749010 & $10^{-5}$ \\
4 & 1.330129946810 & 1.330129946696 & 1.330129946698 & $10^{-7}$\\
5 & 2.832481214762 & 2.832481214761 & 2.832481214761 & $10^{-10}$\\
6 & 3.403909048013 & 3.403909048013 & 3.403909048013 & $10^{-13}$\\
\hline
\end{tabular}

\medskip

\caption{Values of $c_\ell$ computed by Python code with varying internal precision mp.dps}
\label{table:c_ell}
\end{table}

As a 2nd check, we estimated $c_\ell$ an entirely different way,
using the definition of $c_\ell$ given in Theorem \ref{main_theorem} and using Monte Carlo integration.
We took $10^{10}$ random vectors of uniform-$[0,1]$
random variables $(U_1,\ldots,U_{50})$ and
used these to estimate the expectation.  The results
are tabulated in Table \ref{table:MC}.
Of course, one expects deviations from the mean coming from the 
Central Limit Theorem.  But these do appear to confirm at least the first 4 digits of the calculations in Table \ref{table:c_ell}.

\begin{table}[h]
\begin{tabular}{|c|r|}
\hline
$\ell$ & approximate $c_\ell$\\
\hline
1 & 0.1142464511  \\
2 & 0.0032274430 \\
3 & 0.0000314983 \\
\hline
\end{tabular}

\medskip

\caption{Values of $c_\ell$ computed by Monte Carlo methods, 
$10^{10}$ sample vectors}
\label{table:MC}
\end{table}

\medskip

%
%
{\Large \section{Proof of Theorem \ref{clasym}}\label{sec:clasym}}
%
%

We use Proposition \ref{clJ} and invert using the saddle-point method,
as in \S III.5 of \cite{Tenbook}.
By the shape of the binomial distribution, $g(z)$ transitions from
being close to 1 to being very small in the vicinity of
$z= \frac{1}{2\ell}$.
Recall the definition \eqref{E1} of $E_1(z)$ and define
\be\label{ein}
\Ein(s) := \gamma +\log s + E_1(s) = 
 \int_0^s \frac{1-e^{-t}}{t}\, dt,
\ee
which is an entire function of $s$;
see \cite[Theorem 5.9, \S III.5]{Tenbook} for a proof of
the two respresentations in \eqref{ein}.
By  \cite[Theorem 5.10, \S III.5]{Tenbook}, we have
\be\label{rhohat}
\hat{\rho}(s):= \int_{0}^\infty \rho(t) e^{-ts}\, dt = 
e^{\gamma-\Ein(s)}.
\ee

To bound the integral in Proposition \ref{clJ}, we define
\be\label{Jwdef}
J(w,u) := \sum_{k=\ell+1}^\infty 2^{1-k} \binom{k-2}{\ell-1} \Big( E_1(w)-E_1\pfrac{wu}{k} \Big) = E_1(w) -  \sum_{k=\ell+1}^\infty 2^{1-k} \binom{k-2}{\ell-1} E_1\pfrac{wu}{k}.
\ee
In this notation, plus \eqref{ein}, Proposition \ref{clJ} implies that
\be\label{clJ2}
\begin{split}
c_\ell &= \frac{1}{2\pi i u} \int_{\sigma-i\infty}^{\sigma+i\infty}
e^s \exp \big\{ \gamma - \Ein(s/u) + J(s/u,u)  \big\}\, ds\\
&= \frac{1}{2\pi i} \int_{\sigma-i\infty}^{\sigma+i\infty}
e^{uw} \exp \big\{ \gamma - \Ein(w) + J(w,u)  \big\}\, dw,\\
\end{split}
\ee
where $u\ge 1$ is an arbitrary parameter, to be chosen later to make
$J(s/u,u)$ small when $s\approx \sigma$.

Comparing \eqref{clJ2} with \eqref{rhohat}, we will see that
the optimal choise of $u$
is very close to the optimal value needed to compute $\rho(u)$ by inverting $\hat{\rho}$, namely 
\be\label{sigma}
\sigma = -\xi_0 :=   -\xi(u),
\ee
where $\xi=\xi(u)$ satisfies $e^{\xi}=1+u\xi$.  We note that
\be\label{xi-approx}
\xi(u) = \log(u\log u) + \frac{\log\log u}{\log u}
+O\pfrac{(\log\log u)^2}{\log^2 u}.
\ee

We record estimates for $\hat{\rho}(s)$ on vertical segments from 
\cite[Lemma 5.12, Ch. III]{Tenbook}.

\begin{lem}\label{lem:Ten}
Let $u\ge 2$ and $\xi=\xi(u)$.  For $w=-\xi+i\tau$, we have
\[
\hat{\rho}(w) = e^{\gamma-\Ein(w)} = \begin{cases}
O\bigg( \exp \left\{ -\Ein(-\xi) - \frac{\tau^2 u}{2\pi^2}  \right\}
\bigg) & \text{ if } |\tau|\le \pi \\
O \bigg( \exp \left\{ -\Ein(-\xi) - \frac{u}{\pi^2+\xi^2}  \right\}
\bigg) & \text{ if } |\tau|> \pi \\
\frac{1}{w} \( 1 + O\pfrac{1+u\xi}{|w|} \) & \text{ if } |\tau| > 1+ u\xi.
\end{cases}
\]
\end{lem}

We also use a standard bound for the binomial distribution which follows quickly, for example, from Hoeffding's inequality applied to Bernouilli random variables $X_i$
with $\PR(X_i=0)=\PR(X_i=1)=1/2$.

\begin{lem}\label{bintail}
We have
\[
2^{1-k} \binom{k-2}{\ell-1} \ll \exp \left\{ - \frac{(k-2\ell)^2}{2k} \right\}.
\]
\end{lem}

\begin{lem}\label{Al_moments}
Let $A_\ell$ be the random variable with 
\[
\PR(A_\ell=k) = a_{k,\ell} := 2^{1-k} \binom{k-2}{\ell-1} \qquad (k\ge \ell+1).
\]
Then, for $\ell\ge 4$ we have
\begin{enumerate}
\item[(a)] $\E A_\ell = 2\ell+1$;
\item[(b)] $\E |A_\ell-2\ell|^{B} \ll_B \ell^{B/2}$ for all
$B\ge 0$;
\item[(c)] $\ds \E A_\ell^{-1} = \frac{1}{2\ell} + O\pfrac{1}{\ell^3}$;
\item[(d)] $\ds \E A_\ell^{-2} = \frac{1}{4\ell^2} + \frac{1}{8\ell^3} + O\pfrac{1}{\ell^4}$.
\item[(e)] $\ds \E A_l e^{z/A_\ell} \ll \ell e^{z/(2\ell)}$ uniformly for
$0\le z\le \ell^{4/3}$.
\end{enumerate}

\textbf{Remark.}  The random variables are well-defined since
$\sum_k \PR(A_\ell=k) = g(0^+)-g(1/\ell) = 1$.

\begin{proof}
Identity (a) follows from
\[
\E A_\ell = 1 + \E (A_\ell-1) = 1+ \sum_k (k-1)a_{k,\ell} = 1 +
2\ell \sum_k a_{k,\ell+1} = 2\ell+1.
\]
The estimate (b) follows from Lemma \ref{bintail}:
\[
\E |A_\ell-2\ell|^B \ll \sum_{k>\ell} |k-2\ell|^B e^{-\frac{1}{2k}(k-2\ell)^2} \ll \ell^{B/2}.
\]
We prove (c) and (d) in a manner similar to that of the proof of (a).  First, for $k\ge 4$ we have
\[
\frac{1}{k} = \frac{1}{k-2} - \frac{2}{(k-2)(k-3)} + O\pfrac{1}{k^3}
\]
and thus
\dalign{
\E A_\ell^{-1} &= O\pfrac{1}{\ell^3} + \sum_k \( \frac{1}{k-2}  - \frac{2}{(k-2)(k-3)}\) a_{k,\ell} \\
&=  O\pfrac{1}{\ell^3} + \frac{1}{2(\ell-1)} \sum_{k} a_{k,\ell-1}
- \frac{2}{4(\ell-1)(\ell-2)}\sum_k a_{k,\ell-2} \\
&= \frac{\ell-3}{2(\ell-1)(\ell-2)} + O\pfrac{1}{\ell^3}= \frac{1}{2\ell} + O\pfrac{1}{\ell^3}.
}
Similarly,
\dalign{
\E A_\ell^{-2} &= \sum_{k\ge \ell+1} a_{k,\ell} \( \frac{1}{(k-2)(k-3)} -
\frac{5}{(k-2)(k-3)(k-4)} + O\pfrac{1}{k^4} \) \\
&=  O\pfrac{1}{\ell^4}  + \frac{1}{4(\ell-1)(\ell-2)} \sum_{k}
a_{k,\ell-2} - \frac{5}{8(\ell-1)(\ell-2)(\ell-3)} \sum_k a_{k,\ell-3} \\
&= \frac{2\ell-11}{8(\ell-1)(\ell-2)(\ell-3)} +  O\pfrac{1}{\ell^4} \\
&= \frac{1}{4\ell^2} + \frac{1}{8\ell^3} + O\pfrac{1}{\ell^4}.
}
Finally we prove part (e) using Lemma \ref{bintail}.  Let $k_0=\fl{2\ell-10\ell^{2/3}}$
and $k_1=4\ell$.
 We have
\dalign{
\E A_\ell e^{z/A_\ell} &\ll \ell \ e^{z/k_0} + \ell \sum_{k=k_0+1}^{2\ell}
\exp \bigg\{ - \frac{(2\ell-k)^2}{2k} + \frac{z}{k}  \bigg\}
+ \ell \sum_{k>10\ell} \exp \bigg\{ - \frac{(k-2\ell)^2}{2k} + \frac{z}{k}  \bigg\}
\\
&\ll \ell \ e^{z/(2\ell)} + \ell \sum_{k=k_0+1}^{2\ell}
e^{-\ell^{1/3}} + \ell \sum_{k=k_1}^\infty e^{-k/8+z/k_1}
\\
&\ll  \ell \ e^{z/(2\ell)},
}
as required.
\end{proof}
\end{lem}

We use the previous two lemmas to estimate $J(w,u)$, as defined in \eqref{Jwdef}.

%
%
%
%
\begin{prop}\label{Jw}
Suppose that  $u = 2\ell + O(\log \ell)$ and $\xi=\xi(u)$.  Then,
on the vertical line $\Re w = -\xi$ we have
the crude bound
\be\label{Jwcrude}
J(w,u) \ll \frac{e^{\xi}}{|w|} \ll \frac{\ell \log \ell}{|w|}.
\ee
Furthermore, if $|w| \le \ell^{1/4}$ then we have the asymptotic
\be\label{Jwasym}
J(w,u) = e^{-w} \left[ \frac{u-w-1}{2\ell} - 1 +
O(|w|^2 \ell^{-3/2}) \right].
\ee
\end{prop}

\begin{proof}
Using integration by parts, we see that
\be\label{E1E1}
\begin{split}
E_1(w) - E_1\pfrac{w u}{k}  &= \int_1^{u/k} \frac{e^{-wz}}{z}\, dz \\
&= \frac{e^{-w} - e^{-w u/k} (k/u)}{w} 
- \frac{1}{w} \int_{\frac{u}{k}}^1 \frac{e^{-wz}}{z^2}\, dz\\
&\ll \frac{e^\xi + e^{\xi u/k}(k/u)}{|w|} + \frac{(k/u)\max(e^\xi,e^{\xi u/k})}{|w|} \\
&\ll \frac{(e^\xi+e^{\xi u/k})(1+k/u)}{|w|}.
\end{split}
\ee
Apply \eqref{Jwdef}, followed by an application of 
Lemma \ref{Al_moments} (a) and (e).  We have
\dalign{
J(w,u) &\ll \frac{1}{|w|} \sum_{k=\ell+1}^\infty 2^{1-k} \binom{k-2}{\ell-1}
(e^\xi + e^{\xi u/k})(1+k/u) \\
&=\frac{1}{|w|} \E \big( 1 + A_\ell/u \big) \big( e^{\xi} + e^{\xi u/A_\ell} \big)\\
&\ll \frac{\E A_\ell \big( e^{\xi} + e^{\xi u/A_\ell} \big)}{u|w|}  \\
&\ll \frac{\ell e^{\xi} + \ell e^{\xi u/(2\ell)}}{\ell |w|},
}
and \eqref{Jwcrude} follows from the bounds on $u$.

Now suppose that $|w| \le \ell^{1/4}$.
By \eqref{xi-approx}, \eqref{E1E1} and Lemma \ref{bintail}, the terms in
the definition \eqref{Jwdef} of $J(w,u)$ corresponding to
$|k-2\ell| > 100 (\ell \log \ell)^{1/2}$ have total sum
\be\label{bigk}
\ll \frac{e^{2\xi}}{|w|} \sum_{|k-2\ell|>100(\ell \log \ell)^{1/2}}
(1+k/u) a_{k,\ell} \ll \frac{1}{\ell^{100}}.
\ee
When $|k-2\ell| < 100(\ell \log \ell)^{1/2}$, the fraction
$u/k = 1 + O(\sqrt{\frac{\log \ell}{\ell}})$.  Hence
\dalign{
E_1(w) - E_1\pfrac{w u}{k} &= e^{-w} \int_0^{\frac{u}{k}-1} \frac{e^{-wv}}{1+v}\, dv \\
&= e^{-w}  \int_0^{\frac{u}{k}-1} \left(1 - (w+1)v+O(|w|^2 v^2)\right)\, dv\\
&=-e^{-w} \left[ 1 - \frac{u}{k} + (w+1)\(1-\frac{u}{k}\)^2 + 
O\( |w|^2 \frac{|k-u|^3}{\ell^3} \)
\right].
}
By Lemma \ref{Al_moments} (b), 
\[
\E |k-u|^3 \ll \E |k-2\ell|^3 + |2\ell-u|^3 \ll \ell^{3/2}
\]
and thus the big-$O$ term above is $\ll |w|^2 \ell^{-3/2}$.
Reintroducing the summands $|k-2\ell| \ge 100(\ell \log \ell)^{1/2}$, which are negligible by \eqref{bigk}, we find using Lemma \ref{Al_moments} (c) and (d) that
\dalign{
J(w,u) &= O\pfrac{1}{\ell^{100}} -e^{-w}\Bigg[ 1 - u
\E A_\ell^{-1} + (w+1) \E \(1- \frac{u}{A_\ell}\)^2 + O(|w|^2\ell^{-3/2})\Bigg] \\
&= O\pfrac{1}{\ell^{100}} -e^{-w}\Bigg[ 1 - \frac{u}{2\ell}
+(w+1)\( \(1-\frac{u}{2\ell}\)^2 + \frac{u^2}{8\ell^3}\)
+ O(|w|^2\ell^{-3/2})\Bigg]\\
&= e^w \left[ \frac{u-w-1}{2\ell} - 1 +  O(|w|^2\ell^{-3/2})
\right].
}
Here we used repeatedly the bounds $|w|\ge 1$ and $|u-2\ell| \ll \log \ell.$
This completes the proof of \eqref{Jwasym}.
\end{proof}

We now complete the proof of Theorem 2.
Begin with the $w$-integral on the right side of \eqref{clJ2} and define
\be\label{u}
u = 2\ell + 1 - \xi(2\ell), \qquad \sigma = u \xi(u).
\ee
Since
\[
\xi'(u) = \frac{\xi+1}{u(\xi-1)+1} \ll \frac{1}{u}
\]
and $\xi(2\ell)\ll \log \ell$,
it follows that 
\[
\xi(2\ell) = \xi(u) + O\pfrac{\log \ell}{\ell}
\]
and hence that 
\[
u = 2\ell + 1 - \xi(u) + O\pfrac{\log \ell}{\ell}.
\]
Plugging this into \eqref{Jwasym}, we see that
when $w=-\xi+i\tau$ and $|\tau| < \ell^{1/4}$, we have 
the bound
\be\label{smalltau}
J(-\xi+i\tau,u) = e^{-w} \( \frac{-i\tau}{2\ell} +  O(|w|^2\ell^{-3/2}) \)
\ll |\tau| \log \ell + \frac{\log^3 \ell + |\tau|^2\log \ell} {\ell^{1/2}} \quad (|\tau|<\ell^{1/4}).
\ee

We now insert the estimates \eqref{smalltau}, \eqref{Jwcrude}
and the bounds from Lemma \ref{lem:Ten} into
the right side of \eqref{clJ2}.
Let
\[
\tau_1 = 100 \sqrt{\frac{\log u}{u}}, \quad \tau_2 = \pi, \quad 
\tau_3 = 1+u\xi(u).
\]
Write $w=-\xi + i\tau$, $\xi=\xi(u)$.

Our fist task is to show that the part of the integral with $|\tau| > \tau_1$ is negligible.  When $\tau_1\le |\tau| \le \tau_2$,
Lemma \ref{lem:Ten} and \eqref{smalltau} imply that
\dalign{
e^{\gamma-\Ein(w)+J(w,u)} &\ll e^{-\Ein(-\xi)-\tau^2 u/(2\pi^2)+O(|\tau|\log \ell)}\\
&\ll e^{-\Ein(-\xi)-1000\log u}.
}
When $\tau_2 \le |\tau| \le \tau_3$, Lemma \ref{lem:Ten},
\eqref{Jwcrude} and \eqref{smalltau} together imply
\dalign{
e^{\gamma-\Ein(w)+J(w,u)} &\ll e^{-\Ein(-\xi) - \frac{u}{\pi^2+\xi^2}+O(\ell^{3/4}\log \ell)} \\
&\ll e^{-\Ein(-\xi) - \frac{u}{2\xi^2}},
}
and when $|\tau| > \tau_3$, Lemma \ref{lem:Ten} and \eqref{Jwcrude} give
\[
e^{\gamma-\Ein(w)+J(w,u)} = \frac{1}{w} \(1+O\pfrac{\ell \log \ell}{|w|} \).
\]
We find that the portion of the $w$-integral in \eqref{clJ2}
corresponding to $|\tau| \ge \tau_1$ is
\dalign{
&\ll \frac{e^{-u\xi - \Ein(-\xi)}}{\ell^{500}}+
e^{-u\xi}  \int_{\tau_3}^\infty \Bigg|
\frac{e^{ i\tau u}}{\tau}\(1+ O\pfrac{\ell \log \ell}{\tau}\)\ d\tau \Bigg|\\
&\ll  \frac{e^{-u\xi - \Ein(-\xi)}}{\ell^{500}}+e^{-u\xi} \ll
 \frac{e^{-u\xi - \Ein(-\xi)}}{\ell^{500}},
}
upon appealing to the easy bound $-\Ein(-\xi) \gg \xi^{-1} e^{\xi} \gg \ell$.

Finally, we consider $|\tau| \le \tau_1$.  By Lemma \ref{lem:Ten}
and \eqref{Jwcrude} it follows that
\dalign{
\frac{1}{2\pi i} \int_{-\xi-i\tau_1}^{-\xi+i\tau_1} e^{u w}
e^{\gamma-\Ein(w)+J(w,u)}\, dw = K(u) + 
O\Bigg( e^{-u\xi-\Ein(-\xi)}  \frac{\log^2 \ell}{\ell} \Bigg),
}
where
\[
K(u) = \frac{1}{2\pi i} \int_{-\xi-i\tau_1}^{-\xi+i\tau_1} e^{u w}
e^{\gamma-\Ein(w)}\, dw.
\]
Extending the limits to $-\xi \pm i\infty$ produces a small error
term by Lemma \ref{lem:Ten} and it follows from \eqref{rhohat} that
\[
\rho(u) - K(u) \ll
\int_{|\tau| > \tau_1} |e^{uw-\Ein(w)}|\, dw \ll \frac{e^{-\xi -\Ein(-\xi)}}{\ell^{100}}.
\]
Gathering these estimates together, we deduce that
\[
c_l = \rho(u) + O\( \frac{\log^2 \ell}{\ell} e^{-u\xi-\Ein(-\xi)}\).
\]
By Theorem 5.13 of \cite[Ch. III]{Tenbook}, we have
\be\label{rhoasym}
\rho(u) =\(1+O\pfrac{1}{u}\) \pfrac{\xi}{2\pi(u(\xi-1)+1)}^{1/2} e^{\gamma-u\xi-\Ein(-\xi)} \gg
\frac{1}{u^{1/2}} e^{-u\xi-\Ein(-\xi)}
\ee
and thus
\be\label{cl-asym1}
c_l = \rho(u) \( 1 + O\pfrac{\log^2 \ell}{\ell^{1/2}} \).
\ee
Finally, we estimate the error made by replacing $u$ by
\[
u^* = 2\ell + 1 - \log(2\ell \log (2\ell)) - \frac{\log\log (2\ell)}{\log 2\ell}
\]
in \eqref{cl-asym1}.  By \eqref{xi-approx}, 
\[
|u-u^*| \ll \frac{(\log\log \ell)^2}{\log^2 \ell}.
\]
Hence, using \eqref{rhoasym}, \eqref{xi-approx},
the bound $\xi'(u) \ll 1/u$ and the bounds
\dalign{
\Ein(-\xi(u))-\Ein(-\xi(u^*)) &\ll \frac{e^{\xi(u)}}{\xi(u)}| \xi(u^*)-\xi(u)| \ll |u-u^*|,\\
u\xi(u) - u^* \xi(u^*) &\ll |u-u^*| \log u, 
}
we see that
\[
\rho(u) \sim \rho(u^*) \quad (u\to \infty).
\]
Combining this with \eqref{cl-asym1},
this completes the proof of Theorem \ref{clasym}.


\section{Numerical computation of $c$}\label{sec:c}

The terms with $k=1$ and $k=2$ in \eqref{cdef}
contribute 1, respectively,
$\sum_{m=2}^\infty 2^{1-m}\log \pfrac{m}{m-1} = 0.507833922868438392189041\ldots$.
Define
\[
f(t) = \sum_{k=3}^\infty \frac{1}{k!} \;\;\; \idotsint
\limits_{\substack{u_i\ge 0\; \forall i \\u_1+\cdots+u_k=t}} h(u_1)\cdots h(u_k)\, du_1\cdots du_{k-1},
\]
so that $c=f(1)+1.507833922868438392189041\ldots$. 
Extend the definition of $h$ to $(0,\infty)$ by defining $h(u)=1/u$ for $u\ge 1$.  In this way, $h(u)=1/u$ for $u>1/2$, and thus $h$ is $C^\infty$ near $t=1$.
 As in previous sections, define the Laplace transform
\[
F(s) = \int_0^\infty f(t) e^{-st}\, dt = e^J-1-J^2/2. \quad J = 
\int_0^\infty h(u) e^{-su}\, du.
\]
Using that $h(u)=u^{-1} 2^{1-m}$ for $\frac{1}{m+1}<u\le \frac{1}{m}$, $m\ge 1$, and recalling the definition 
\eqref{E1} of $E_1(z)$, we quickly derive
\dalign{
\int_0^\infty h(u) e^{-su}\, du &=
\sum_{m=1}^\infty 2^{1-m} \int_{1/(m+1)}^{1/m} \frac{e^{-su}}{u}\, du + \int_1^\infty \frac{e^{-su}}{u}\, du \\
&= \sum_{m=2}^\infty 2^{1-m} E_1(s/m).
}
Again, we use the Python package \texttt{mpmath}
to numerically invert the Laplace transform $F(s)$,
and this gives $c=f(1)=1.526453\ldots$.

%
%


\appendix

\section{Proof of Lemma \ref{DG}}

Recall that for random $\qq=\qq(n)=(q_1,\ldots,q_k)$ we defined
\be\label{Xin}
X_i(n) = \frac{\log q_i}{\log(\frac{n}{q_1\cdots q_{i-1}})}.
\ee
It suffices to show that for any real numbers $0<a_i<b_i<1$ $(1\le i\le k$), 
\be\label{uniform}
\PR_x ( a_i \le X_i(n) \le b_i \; (1\le i\le k) ) \to \prod_{i=1}^k
(b_i-a_i) \qquad (x\to \infty).
\ee
Below, constants implied by $O-$ an $\ll-$ may depend on $k$
and the $a_i,b_i$.
From \eqref{qiX}, if $X_{i}\le b_i$ for all $i$ then
\be\label{frac-lwr}
\frac{n}{q_1\cdots q_{i-1}} \ge n^{(1-b_1)\cdots (1-b_{i-1})}.
\ee
Hence, writing $c=(1-b_1)\cdots (1-b_k)\min_i a_i$,
we have $q_i > n^{c}$ for all $i$ under the assumption 
that $a_i \le X_i(n) \le b_i$ for every $i$.  If some $q_i$ 
is not prime, then $n$ is divisible by a prime power $p^a>x^{c/2}/\log x$
with $a\ge 2$ and the number of such $n\in (x,2x]$ is
$O(x^{1-c/2})$.  Thus, we may assume that the $q_i$
are all prime.  
In this case, $\log q_=\Lambda(q_i)$ and hence $X_i(n)$ equals the probability that $q_i$ is chosen at step $i$.
We calculate, using \eqref{Xin},
\dalign{
\PR_x ( a_i &\le X_i(n) \le b_i \; (1\le i\le k) ) = \frac1{x} \sum_{x<n\le 2x} \ssum{q_1|n \\ a_1\le X_1(n) \le b_1} X_1(n)
\cdots \ssum{q_k|n \\ a_1\le X_k(n) \le b_1} X_k(n).
}
On the right side, the variables $q_i$ are no longer random, 
but we still define $X_i(n)$ by \eqref{Xin}.
Since $\log x \le \log n\le \log (2x)$, the above expression is bounded below by
\dalign{
(1+O(1/\log x)) \sum_{a_1 \log (2x) \le \log q_1 \le b_1 \log x}
\frac{\log q_1}{q_1}  \cdots \sum_{a_k \log(\frac{2x}{q_1\cdots q_{k-1}}) \le \log q_k \le b_k \log(\frac{x}{q_1\cdots q_{k-1}})}
\frac{\log q_k}{\log \frac{x}{q_1\cdots q_{k-1}}},
}
and bounded above by the same expression with ``$x$'' and
``$2x$'' interchanged in the logarithms.

For each fixed $q_1,\ldots,q_{i-1}$, Mertens' estimate gives
\[
\sum_{a_i \log(\frac{x}{q_1\cdots q_{i-1}})+O(1) \le \log q_i \le b_i \log(\frac{x}{q_1\cdots q_{i-1}})+O(1)}
\frac{\log q_i}{\log \frac{x}{q_1\cdots q_{i-1}}} =
b_i-a_i + O\pfrac{1}{\log x},
\]
and the desired result \eqref{uniform} follows.



\begin{thebibliography}{99}

\bibitem{Bi} P. Billingsly, {\it
On the distribution of large prime divisors},
Collection of articles dedicated to the memory of Alfr\'ed R\'enyi, I.
Period. Math. Hungar. {\bf 2} (1972), 283--289.


\bibitem{Dav} H. Davenport, {\it Multiplicative number theory}, 3rd ed.,
Graduate Texts in Mathematics vol. 74, Springer-Verlag, New York, 2000.


\bibitem{DG} P. Donnelly and G. Grimmett, {\it On the asymptotic
    distribution of large prime factors}, J. London Math. Soc.
(2) {\bf 47} (1993), 395--404.

\bibitem{EGRS}
P. Erd\H os, R. L. Graham, I. Z. Ruzsa, E. G. Straus, 
{\it On the prime factors of  $\binom{2n}{n}$},
 Collection of articles in honor of Derrick Henry Lehmer on the occasion of his seventieth birthday, Math. Comp. {\bf 29} (1975)
83--92.



\bibitem{GK} S. W. Graham and G. Kolesnik, {\it Van der Corput's method of exponential sums}, London Math. Soc. Lecture Note, vol. 126, Cambridge University Press, 1991.


\bibitem{Gran}
A. Granville, {\it Arithmetic properties of binomial coefficients. I. Binomial coefficients modulo prime
powers}, Organic Mathematics (Burnaby, BC, 1995), 253--276, CMS Conf. Proc., 20, Amer. Math. Soc., Providence, RI, 1997.

\bibitem{HR} 
H.~Halberstam and H.-E.~Richert, {\it Sieve Methods\/}, Academic
Press, London, 1974.


\bibitem{HT}
R. R. Hall, G. Tenenbaum, {\em Divisors,} Cambridge
Tracts in mathematics vol. {\bf 90}, 1988.

\bibitem{Handa} K. Handa, {\it The two-parameter Poisson-Dirichlet process}, Bernoulli \textbf{15} (2009), 1082--1116.

\bibitem{K} E. E. Kummer, {\it \"Uber die Erg\"anzungss\"atze zu den allgemeinen Reciprocit\"atsgesetzen}, J. Reine angew. Math.
{\bf 44} (1852), 93--146.

\bibitem{HLM} H. L. Montgomery, {\it Ten Lectures on the Interface Between
Analytic Number Theory and Harmonic Analysis}, CBMS Regional Conference
Series in Mathematics vol. 84, Amer. Math. Soc., 1994.


\bibitem{OEIS} N. A. Sloane,
{\it The on-line encyclopedia of integer sequences. }
\texttt{http://oeis.org}



\bibitem{Pom} C. Pomerance, {\it Divisors of the middle binomial 
coefficient}, Amer. Math. Monthly {\bf 122} (2015), 636--644.


\bibitem{sanna} C. Sanna,
{\it Central binomial coefficients divisible by or coprime to their indices.} 
Int. J. Number Theory {\bf 14} (2018), no. 4, 1135--1141. 


\bibitem{Ten} G. Tenenbaum, {\it A rate estimate in Billingsley's
theorem for the size distribution of large prime factors},
  Quart. J. Math. Oxford  {\bf 51}  (2000),  no. 3, 385--403.

\bibitem{Tenbook} G.~Tenenbaum,   {\it Introduction to Analytic and
Probabilistic Number Theory, 3rd ed.}, Amer. Math Soc., 2015
%
%

\end{thebibliography}
\end{document}